\documentclass[12pt,a4paper]{article}
\usepackage[
  margin=2.6cm,
  includefoot,
  bmargin=1.3cm,
  footskip=1.3cm
]{geometry}

\usepackage{float}
\immediate\write18{./sync_bibliography.sh}
\usepackage[T1]{fontenc}
\usepackage{lmodern}

\usepackage{textgreek}
\usepackage{mathtools}

\newcommand{\supth}{\textsuperscript{th}\ }
\usepackage{amssymb} 
\usepackage{amsmath,mathtools} 
\usepackage{amsthm}
\theoremstyle{plain}
\newtheoremstyle{slanted}
  {}
  {}
  {\slshape}
  {}
  {\bfseries}
  {}
  { }
  {}
\theoremstyle{slanted}
\newtheorem{thm}{Theorem}
\newtheorem{lem}{Lemma}
\newtheorem{prop}{Proposition}

\newtheorem{cor}{Corollary}
\theoremstyle{definition}
\newtheorem{defn}{Definition}
\newtheorem{obs}{Observation}

\theoremstyle{remark}

\usepackage{xcolor}%
\usepackage[colorlinks]{hyperref}
\usepackage{breakurl}
\hypersetup{
    colorlinks=true,       
    linkcolor=cyan!15!black,          
    citecolor=cyan!15!black,        
    filecolor=cyan!15!black,      
    urlcolor=cyan!15!black,           
    runcolor=cyan!15!black
}
\usepackage{verbatim}
\usepackage{graphicx}
\usepackage{booktabs}
\usepackage[font={small}]{caption}
\usepackage{dcolumn}
\usepackage{stackengine}
\usepackage{setspace}
\usepackage{algpseudocode,algorithm,algorithmicx}
\floatstyle{plain}
\newfloat{myalgo}{tbhp}{mya}
{\begin{myalgo}[#1]
\centering
\begin{minipage}{#2}
\begin{algorithm}[H]\setstretch{1.15}}%
{\end{algorithm}
\end{minipage}
\vspace{0.5cm}
\end{myalgo}}
\usepackage{physics}
\usepackage{complexity}

\usepackage{enumerate}
\usepackage{enumitem}
\usepackage{hhline}
\usepackage{array}
\newcolumntype{P}[1]{>{\centering\arraybackslash}p{#1}}
\usepackage{fancyhdr}
\usepackage{todonotes}
\usepackage{chngcntr}

\newcommand{\wprod}{\ensuremath{\mathbin{\nabla}}}

\begin{document}
\title{Perfect weak modular product graphs}
\author{Danial Dervovic\thanks{Department of Computer Science, University College London, Gower Street, London WC1E 6BT, United Kingdom; mailto:  \href{mailto:d.dervovic@cs.ucl.ac.uk}{\texttt{d.dervovic@cs.ucl.ac.uk}}.}}
\date{\today}
\maketitle

\begin{abstract}
In this paper we enumerate the necessary and sufficient conditions for the weak modular product of two simple graphs to be perfect.
The weak modular product differs from the direct product by also encoding non-adjacencies of the factor graphs in its edges.
This work is motivated by the following: a 1978 theorem of Kozen states that two graphs on $n$ vertices are isomorphic if and only if there is a clique of size $n$ in the weak modular product between the two graphs.
Furthermore, a straightforward corollary of Kozen's theorem and Lov\'{a}sz's sandwich theorem is if the weak modular product between two graphs is perfect, then checking if the graphs are isomorphic is polynomial in $n$.
Interesting cases include complete multipartite graphs and disjoint unions of cliques. All perfect weak modular products have factors that fall into classes of graphs for which testing isomorphism is already known to be polynomial in the number of vertices. 
\end{abstract}

\section{Introduction}

Graph products have been extensively studied and are of vast theoretical and practical interest, see, e.g., Hammack, Imrich and Klav\v{z}ar~\cite{Hammack2011}.
A common problem is to determine how graph invariants such as the independence number and clique number behave under the action of a particular graph product.
For instance, a famous result of Lov\'{a}sz~\cite{Lovasz1979} states that for graphs $G$ and $H$, $\vartheta(G \boxtimes H) = \vartheta(G) \vartheta(H)$, where $\vartheta(\,\cdot\,)$ denotes the Lov\'{a}sz number and $\boxtimes$ is the strong graph product.
There are three graph products that have received most attention in the literature: the aforementioned strong graph product, the direct (or tensor) product and the Cartesian product.
These three graph products are the most studied as they satisfy the following: they are associative and projections onto the factors are weak homomorphisms.
Loosely speaking, the second property means that the adjacency structure of the product graph allows one to approximately infer the adjacency structure of the factor graphs.

One can consider graph products that do not satisfy these two properties.
One such product is the \emph{weak modular product}, whose adjacency structure also includes information about non-adjacency in the factor graphs.
Interestingly, as originally proved by Kozen~\cite{Kozen1978}, a clique of a certain size exists in the product graph if and only if the factors are isomorphic.
Since the decision version of finding the clique number of a general graph is $\NP$-complete, this result has largely been ignored in the literature with reference to graph isomorphism.

The graph isomorphism problem (GI)
has been studied extensively for decades, but its complexity status remains unknown.
Clearly, $\mathrm{GI}\in\NP$ as one can check easily if a given candidate isomorphism preserves all adjacencies and nonadjacencies between the two graphs at hand.
However, it is unlikely that GI is $\NP$-complete as this would imply collapse of the polynomial hierarchy~\cite{Goldreich1991}.
The question of whether $\mathrm{GI}\in\P$ remains open.
There has been a considerable research effort to find a polynomial-time algorithm for GI, culminating recently in the recent quasi-polynomial
 algorithm by Babai~\cite{Babai2015}.
However, there are many classes of graph for which GI admits a polynomial-time algorithm, for instance, graphs with a forbidden minor~\cite{Ponomarenko1991,Grohe2010}, including planar graphs and graphs of bounded genus.
In practice, the approach of McKay and Piperno~\cite{McKay2014} works efficiently on almost all graphs and so efficiently solving GI ``in the wild'' is all but solved.

Many of the recent advances in GI, including Babai's recent breakthrough~\cite{Babai2015} and the \texttt{nauty}/\texttt{traces} programs of McKay and Piperno~\cite{McKay2014} use a group theoretic approach.
In this paper we consider a combinatorial approach, which was the primary method for GI in the earlier days of its study.
The combinatorial construction we consider is the weak modular product, mentioned earlier.
This construction has been used in the pattern recognition community under the label \emph{association graph} to solve graph matching problems~\cite{Pelillo1999a, CONTE2004}.
Indeed, Pelillo~\cite{Pelillo1999} uses a heuristic inspired by theoretical biology to find cliques in the weak modular product as an approach to inexact graph matching, a problem that can be interpreted as an approximation to GI. He provides computational evidence that this technique is tractable for this problem in certain regimes.

We take an analytical approach, inspired by the following observation: a direct corollary of the Lov\'{a}sz sandwich theorem gives us that the Lov\'{a}sz number of a perfect graph is the same as its clique number~\cite{Lovasz1979}.
Since the Lov\'{a}sz number can be computed in polynomial time, if the weak modular product of two graphs is perfect, testing if they are isomorphic is polynomial.
Note that GI for perfect graphs is GI-complete, since deciding if two bipartite graphs are isomorphic is GI-complete~\cite{Uehara2005} and all bipartite graphs are perfect.
We enumerate all pairs of graphs for which the weak modular product is perfect, using theoretical tools that were not available to Kozen in 1978, including the Strong Perfect Graph Theorem~\cite{Chudnovsky2006} amongst others.
This adds to a tradition of enumerating perfect product graphs; Ravindra and Parthasarathy~\cite{Ravindra1977} and Ravindra~\cite{Ravindra1978} found all perfect Cartesian, direct and strong product graphs.  
This results in the following theorem.
\begin{thm}\label{thm:main_intro}
    The graph $G = G_0 \wprod G_1$ is perfect if and only if one of the following holds:
    \begin{enumerate}
     \item $G_z \in \{K_1, K_2, E_2\}$, $G_{\overline{z}}$ arbitrary;
     \item $G_z \cong P_4$, $G_{\overline{z}} \in \{ K_{1,r}, K_{r} \uplus K_1, P_4 \}$;
     \item $G_z \cong C_5$, $G_{\overline{z}} \in \{ P_3, K_2 \uplus E_1, P_4, C_5 \}$;
     \item $G_z \cong K_r \uplus K_s$, $G_{\overline{z}}$ is a disjoint union of stars and cliques;
     \item $G_z \cong K_{m,n}$, $G_{\overline{z}}$ is connected and $(P_4, \text{cricket}, \text{dart}, \text{hourglass})$-free;
     \item $G_z \cong K_n$, $G_{\overline{z}}$ $(\text{odd hole, paw})$-free;
     \item $G_z \cong E_n$, $G_{\overline{z}}$ $(\text{odd antihole, co-paw})$-free;
     \item $G_z$, $G_{\overline{z}}$ are complete multipartite;
     \item $G_z$, $G_{\overline{z}}$ are disjoint unions of cliques;
     \item $G_z \cong K_r \uplus K_s$, $G_{\overline{z}} \cong K_{m,n}$;
    \end{enumerate}
    for any $m,n,r,s,z$, where $m,n,r,s \in \mathbb{N}$, and $z \in \{0,1\}$, with its (Boolean) negation denoted by $\overline{z}$.
\end{thm}
The remainder of this work constitutes the proof and necessary ingredients.

\section{Preliminaries}\label{sec:prelim}

We consider only finite, simple graphs, i.e. graphs with finite number of vertices and no self-loops or multiple edges.
A \emph{graph} $G=(V(G),E(G))$ consists of a set $V(G)$ of $n$ vertices and a set of edges $E(G)\subseteq \{ \{ x,x'\} : x, x'\in V(G),\ x\neq x'\}$.

We write $x\sim y$ to denote that vertices $x$ and $y$ are adjacent.



We denote by $\overline{G}$ the \emph{complement} of $G$, $V(\overline{G})=V(G)$ and $x\sim x'$ in $\overline{G}$ if and only if $x \not\sim x'$ in $G$ and $x \neq x'$.
For a named graph, we prepend the prefix \emph{``co-''} to denote its complement, e.g., the complement of a bipartite graph is a co-bipartite graph.  
The \emph{union} of graphs $G$ and $H$ is the graph $G\cup H$ with vertex set $V(G\cup H)=V(G) \cup V(H)$ and edge set $E(G\cup H)=E(G)\cup E(H)$.
The \emph{disjoint union} of graphs $G$ and $H$ is the graph $G\uplus H$ with vertex set $V(G\uplus H)=V(G)\uplus V(H)$ and edge set $E(G\uplus H)= E(G)\uplus E(H)$.
For a graph $G$, we denote the disjoint union of $k$ copies of $G$ by $k G$.
\par
A graph $G'$ is a \emph{subgraph} of another graph $G$, $G'\subseteq G$ , if and only if $V(G')\subseteq V(G)$, and
\begin{equation}
    E(G')\subseteq E(G) \text{ and for all } \{x, x'\}\in E(G'), \ \ x, x'\in V(G').
\end{equation}
Suppose we have a subset of the vertices $U\subseteq V(G)$.
An \emph{induced subgraph} of $G$, $G[U]$, is the graph with vertex set $V(G[U])=U$ and edge set
\begin{equation}
    \{ \{ x, x'\}\mid x, x' \in U,\ \{ x, x'\}\in E(G) \}.
\end{equation}
We say $G[U]$ is \emph{induced} by $U\subseteq V(G)$.

We now define two closely related graph products.
The \emph{direct}, or \emph{tensor product} of graphs $G$ and $H$, denoted by $G\otimes H$, has vertex set $V(G)\times V(H)$ and an edge $\{ (x,y),(x',y')\}$ if and only if $\{ x,x'\}\in E(G)$ and $\{ y,y'\}\in E(H)$.
\par
The \emph{weak modular product} (see, e.g., Hammack, Imrich and Klav\v{z}ar~\cite{Hammack2011}) of graphs $G$ and $H$, denoted by $G\wprod H$, has vertex set $V(G\wprod H) =V(G)\times V(H)$ and an edge $\{ (x,y),(x',y')\}$ if and only if
\begin{enumerate}
    \item either $\{ x,x'\}\in E(G)$ and $\{ y,y'\}\in E(H)$;
    \item or $\{ x,x'\} \in E(\overline{G})$ and $\{ y,y'\} \in E(\overline{H})$.
\end{enumerate}
\par
The next statement is a direct consequence of the definitions of the weak modular product and the tensor product.
\begin{lem}\label{lem:GprodHunion}
    For graphs $G$ and $H$,
    \begin{equation}
        G\wprod H =G \otimes H \cup \overline{G}\otimes \overline{H}.
    \end{equation}
\end{lem}

Given graphs $G$ and $H$, we say that $G$ and $H$ are \emph{isomorphic}, $G\cong H$, whenever there is a bijection $f:V(G)\to V(H)$ such that $x \sim y$ if and only if $f(x)\sim f(y)$ for every $x,y \in V(G)$.
A \emph{clique} is a subset of the vertices of a graph such that every two distinct vertices in the clique are adjacent.
The \emph{clique number} of a graph $G$, $\omega(G)$, is the cardinality of its largest clique.

An \emph{independent set} in a graph is a subset of the vertices such that no two vertices in the subset are adjacent.
The \emph{independence number} of a graph $G$, $\alpha(G)$, is the cardinality of its largest independent set.
Clearly, $\alpha(G) = \omega(\overline{G})$.

The \emph{chromatic number} of a graph $\chi(G)$ is the minimum number of colours for which every pair of adjacent vertices has a different colour when we give every vertex a colour.

A graph is \emph{perfect} if the chromatic number of every induced subgraph equals its clique number.
Lov\'{a}sz's famous `sandwich theorem' states that for any graph $G$, $\omega(G)\leq \vartheta(G) \leq \chi(G)$, where $\vartheta(G)$ is the Lov\'{a}sz number of the graph $G$, which can be computed in polynomial time~\cite{Lovasz1979}.
Thus, for perfect graphs one can compute $\omega(G)$ and $\chi(G)$ in polynomial time.

\section{The weak modular product and isomorphism}\label{sec:wmp_and_iso}

For completeness we state and prove Kozen's theorem, in modern language.

\begin{prop}\label{prop:Kozen}
  \emph{(Kozen~\cite{Kozen1978}).}
  Let $G$ and $H$ be graphs on $n$ vertices.
  Then $\omega(G\wprod H)\leq n$.
  Moreover, $\omega(G\wprod H) = n$ if and only if $G\cong H$.
\end{prop}
\begin{proof}
To see that there is no clique in $G\wprod H$ larger than $n$ consider the following.
First lay the vertices of $G\wprod H$ in an $n\times n$ grid so that the vertex $(x,y)$ is in the same row as $(x',y')$ if $x=x'$, and in the same column if $y=y'$.
Then by the definition of the weak modular product there can be no edges between vertices in the same row or in the same column.
The vertices of an $n$-clique thus will occupy positions on the grid such that no two vertices are in the same row or column.
No larger clique can exist since there is no position in the grid where one can place a new vertex such that it does not share a row or column with any of the vertices already in the clique.

Now suppose there is an $n$-clique in $G\wprod H$.
The vertices $(x,y)$ in the clique represent the bijection $x\mapsto y$ for all $x\in V(G)$, $y\in V(H)$, which we denote $\sigma$.
We can see that $\sigma$ is an isomorphism because for all $x, x' \in V(G)$, $\sigma(x)\sim \sigma(x')$ if and only if $x \sim x'$, from the definition of the weak modular product.
For the converse, suppose that $G\cong H$, with $\sigma : V(G)\to V(H)$ an isomorphism.
Then from the definition of the weak modular product, we will have the collection of edges
\begin{equation*}
\left\{ \{ (x, \sigma(x)), (x', \sigma(x'))\} \,\middle|\,  x, x' \in V(G),\ x' \neq x \right\} \subseteq
E(G\wprod H).
\end{equation*}
This collection of edges induces an $n$-clique in $G\wprod H$ from the definitions of the weak modular product and isomorphism, so an $n$-clique exists if and only if $G\cong H$.
\end{proof}

Thus, for two graphs on $n$ vertices $G$ and $H$, deciding if an $n$-clique exists in $G \wprod H$ is equivalent to deciding if $G \cong H$.
It is well known that computing the clique number of a perfect graph is polynomial-time in $n$ for perfect graphs, via the Lov\'{a}sz sandwich theorem as discussed in Section~\ref{sec:wmp_and_iso}.
For a pair of graphs $(G,H)$, if $G \wprod H$ is perfect, deciding if $G \cong H$ is polynomial-time in $n$.

\section{Perfect weak modular products}

In this section, we prove Theorem~\ref{thm:main_intro}, namely we enumerate the pairs $(G,H)$ for which $G \wprod H$ is perfect. We will need some further definitions and results.

\subsection{Classes of graphs}

We now provide definitions and characterisations of families of graphs that will be of use later.

We use standard notation for named graphs; for instance, the complete, empty and path graphs on $n$ vertices are denoted by $K_n$, $E_n$ and $P_n$ respectively.

A graph $G$ on $n$ vertices is said to be \emph{bipartite} if $V(G)= V_0(G)\cup V_1(G)$ such that if $x \sim x'$ then $x \in V_0(G)$ and $x' \in V_1(G)$ or $x \in V_1(G)$ and $x' \in V_0(G)$, for every $x, x' \in V(G)$.
The sets $V_0(G)$ and $V_1(G)$ are said to be the \emph{partite sets} of $G$.
A \emph{complete bipartite} graph $K_{m,n}$, where $\abs{V_0(G)} = m$, $\abs{V_1(G)} = n$, is a bipartite graph $G$ for which every vertex in $V_0(G)$ is connected to every vertex in $V_1(G)$.
A \emph{star} is a complete bipartite graph where at least one of the partite sets has only one vertex.
Bipartite graphs are well known to be perfect.
A \emph{complete multipartite} graph $K_{n_1, n_2, \ldots, n_k}$ is defined similarly, with the relaxation that there are now $k \geq 2$ partite sets. Any induced subgraph covering two different partite sets of a complete multipartite graph is complete bipartite.

The \emph{diamond} or $K_{1,1,2}$, \emph{paw} or $Y$, \emph{cricket}, \emph{dart} and \emph{hourglass} graphs are defined in Figure~\ref{fig:cricket_dart_hourglass}.
\begin{figure}[H]
    \centering
    \includegraphics[width=0.85\textwidth]{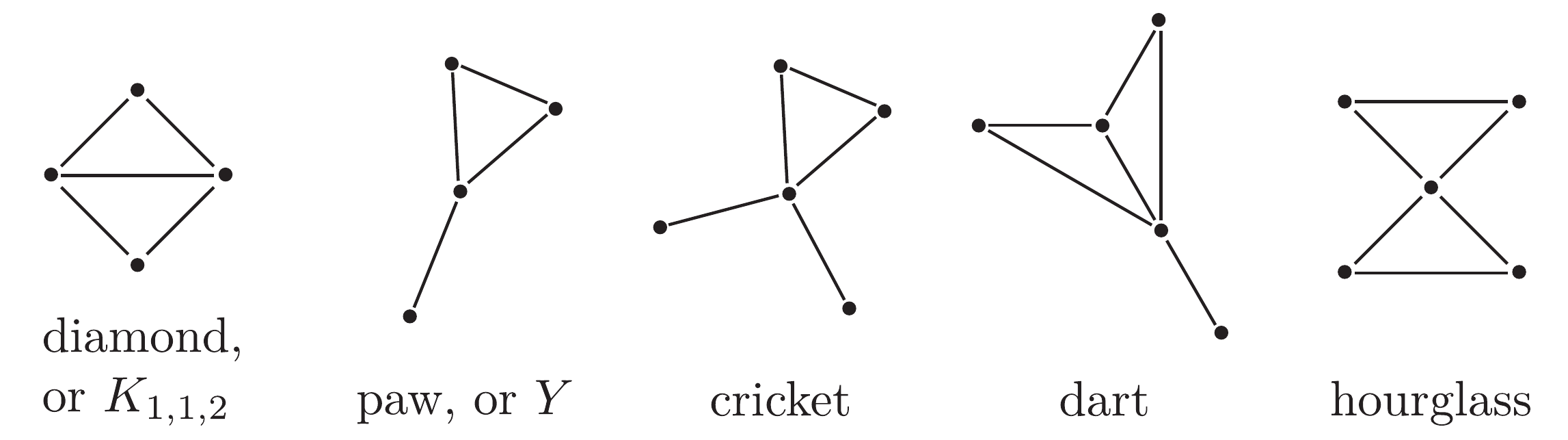}
    \caption{Some named graphs.}
    \label{fig:cricket_dart_hourglass}
\end{figure}
\begin{obs}\label{obs:namedcomplements}
    The complements of the diamond, paw, cricket, dart and hourglass are respectively: $K_2 \uplus E_2$, $P_3 \uplus E_1$, $K_{1,1,2} \uplus E_1$, $Y \uplus E_1$ and $C_4 \uplus E_1$.
\end{obs}

\begin{lem}\label{lem:triangle_free_odd_hole}
    Let $G$ be nonbipartite and triangle-free. Then, $G$ has an induced odd cycle of order $\geq 5$.
\end{lem}
\begin{proof}
    Since $G$ is not bipartite it contains an odd cycle. Moreover, $G$ is triangle-free so this odd cycle must have order $\geq 5$. Call $C$ the smallest odd cycle in $G$. Consider a chord $e$ in $C$: since $C$ has an odd number of vertices, the subgraph induced on the union of $C$ and $e$ contains two cycles: an even cycle and an odd cycle. This gives a contradiction since the odd cycle is smaller than $C$, but $C$ is the smallest odd cycle in $G$. Thus, there is no chord in $C$ and $C$ is an induced subgraph.
\end{proof}

\begin{lem}\label{lem:p3free}
    A graph $G$ is a disjoint union of cliques if and only if it has no induced $P_3$. 
\end{lem}
\begin{proof}
    Suppose $G$ is the disjoint union of cliques. Then, every induced subgraph on 3 vertices is either: $K_3$, $K_2 \uplus E_1$, or $E_3$, so clearly is $P_3$-free. Now suppose $G$ contains $P_3$ as an induced subgraph. Then, we have two vertices in the same connected component that are not connected, and so $G$ is not the disjoint union of cliques.
\end{proof}

The following lemma is a direct consequence of the relevant definitions.
\begin{lem}\label{lem:disjoint_compl_multi}
    A graph $G$ is complete multipartite if and only if its complement is a disjoint union of cliques. 
\end{lem}

\begin{cor}\label{cor:compmulti_K2uE1free}
    A graph $G$ is complete multipartite if and only if it is $(K_2 \uplus E_1)$-free.
\end{cor}
\begin{proof}
    Combine Lemma~\ref{lem:p3free} with Lemma~\ref{lem:disjoint_compl_multi}.
\end{proof}

\begin{lem}\label{lem:compbipP4}
     Any connected bipartite graph $G$ that is not complete has an induced $P_4$. 
\end{lem}
\begin{proof}
    Let $u \in V_0(G)$ and $v\in V_1(G)$. The shortest path $P(u,v)$ between $u$ and $v$ is odd length. If the length of this path is $1$ for all pairs $(u,v)$, $G$ is complete bipartite; else, the length of $P(u,v)$ is $3$ or greater for some pair $(u,v)$, and so we have an induced $P_{2k}$ for some $k \geq 2$, proving the lemma. 
\end{proof}

\begin{lem}\label{lem:diamondfree_comp_multip}
    A complete multipartite graph $G$ is diamond-free if and only if it is a clique or complete bipartite.
\end{lem}
\begin{proof}
    $(\Rightarrow)$ We prove the contrapositive. Suppose $G$ is complete multipartite and not bipartite and not a clique. $G$ must have at $k \geq 3$ partite sets, for if it has two partite sets it would be complete bipartite. Now if every partite set in $G$ has one vertex then $G \cong K_k$, but $G$ is not a clique by assumption. So, $G$ has an induced $K_{1,1,2}$, a diamond.

    $(\Leftarrow)$ If $G$ is a clique then any induced subgraph on 4 vertices is isomorphic to $K_4 \not\cong K_{1,1,2}$. Bipartite graphs are triangle-free so are trivially diamond-free.
\end{proof}

\begin{lem}[Paw-free graphs~{\cite[Theorem 1]{Olariu1988}}]
    \label{lem:paw_free_graphs}
    A graph $G$ is a paw-free graph if and only if each component of $G$ is triangle-free or complete multipartite. 
\end{lem}

\begin{lem}\label{lem:p4pawfree}
    If $G$ is connected and $(P_4,\text{paw})$-free, then it is complete multipartite.
\end{lem}
\begin{proof}
    From Lemma~\ref{lem:paw_free_graphs}, if $G$ is connected and paw-free, then it is complete multipartite or triangle-free. If $G$ is triangle-free then it is bipartite, for otherwise it has an odd hole, but $G$ is $P_4$-free and every odd hole has an induced $P_4$. So, $G$ must be bipartite if it is triangle-free. Moreover, since $G$ is $P_4$-free, if it is bipartite then it is complete bipartite from Lemma~\ref{lem:compbipP4}.
\end{proof}

\begin{lem}\label{lem:disjunstrscliqcpml}
    A graph $G$ is a disjoint union of cliques and stars with two or more connected components if and only if its complement $\overline{G}$ is connected and $(P_4, \text{cricket, dart, hourglass})$-free.
\end{lem}
\begin{proof}
    $(\Rightarrow)$ If $G$ is a disjoint union of two or more cliques and stars, $\overline{G}$ is connected, by the following: suppose $G$ has $k \geq 2$ connected components, i.e. $V(G) := V_1 \uplus \cdots \uplus V_k$ where the $V_i$ are pairwise disconnected.
    If we have $x \in V_i$, $x' \in V_j$ for $i \neq j$, then $x \sim_{\overline{G}} x'$.
    If $x,x' \in V_i$, choose any $x'' \in V_j$ for $i \neq j$ and $(x,x'',x')$ is a path in $\overline{G}$.
    Since there is a path between any $x,x' \in V(\overline{G})$, $\overline{G}$ is connected.
    We have that $\overline{G}$ is $P_4$-free, since a disjoint union of cliques is $P_3$-free, disjoint unions of complete bipartite graphs are $P_4$-free by Lemma~\ref{lem:compbipP4} and $P_4$ is self-complementary. Now, $\overline{G}$ is $(\text{cricket, dart, hourglass})$-free if and only if $G$ is $(K_{1,1,2}\uplus E_1, Y \uplus E_1, C_4 \uplus E_1 )$-free, which is necessarily true if $G$ is $(\text{diamond}, \text{paw}, K_{2,2})$-free.
    This is satisfied when $G$ is a disjoint union of stars and cliques, by considering each of the components piecewise with Lemmas~\ref{lem:paw_free_graphs}~and~\ref{lem:compbipP4}.

    $(\Leftarrow)$ We prove the contrapositive. Suppose $G$ has one connected component. Then, $\overline{G}$ is disconnected or contains an induced $P_4$, since by~{\cite[Aux. Thm.]{Seinsche1974}},~\cite{Lerchs72}, the complement of a connected $P_4$-free graph is disconnected.
    Now suppose $G$ has two or more components. If $G$ has an induced paw then it also has an induced $Y \uplus E_1$ and so $\overline{G}$ contains an induced dart by Observation~\ref{obs:namedcomplements}.
    If $G$ is paw-free then $G$ is a disjoint union of complete multipartite graphs by Lemma~\ref{lem:p4pawfree}.
    Let $G$ be paw-free.
    If $G$ has an induced diamond then it contains an induced $K_{1,1,2}\uplus E_1$ and so $\overline{G}$ contains an induced cricket. If $G$ has no induced diamond then it is a disjoint union of cliques and complete bipartite graphs by Lemma~\ref{lem:diamondfree_comp_multip}. If $G$ contains an induced $K_{2,2} \cong C_4$ then it has an induced $C_4 \uplus E_1$ and so $\overline{G}$ has an induced hourglass by Observation~\ref{obs:namedcomplements}.
    We are left with $G$ being the disjoint union of two or more stars and cliques, which we have already proved to be $(P_4, \text{cricket, dart, hourglass})$-free and connected, so the result is proven.
\end{proof}

\begin{lem}
    Let $G$ be a complete multipartite graph. Then it is connected and\\$(P_4, \mathrm{dart, cricket, hourglass})$-free.
\end{lem}
\begin{proof}
    Trivially, $G$ is connected. Now, observe that every induced subgraph of a complete multipartite graph is complete multipartite. Equivalently, if $G$ has an induced subgraph that is not complete multipartite, $G$ is not complete multipartite. If $G$ contains an induced $X \in \{P_4, \mathrm{dart, cricket, hourglass}\}$, then we have a contradiction, since $X$ is not complete multipartite.
\end{proof}

\begin{lem}\label{lem:pawfree_perfect}
    A graph $G$ is $(\text{odd-hole, paw})$-free if and only if each component of G is bipartite or complete multipartite.
\end{lem}
\begin{proof}
    $(\Rightarrow)$ If $G$ is paw-free then by Lemma~\ref{lem:paw_free_graphs} each component is triangle-free or complete multipartite. Let $X$ be any triangle-free component of $G$. Since $X$ is odd-hole free, $X$ is bipartite by Lemma~\ref{lem:triangle_free_odd_hole}.
    
    $(\Leftarrow)$ Let $X$ be a given component of $G$. Suppose \emph{i.} $X$ is bipartite. By definition, $X$ has no odd hole since it contains no odd cycles. Moreover, the paw contains a triangle as an induced subgraph so it cannot be an induced subgraph of $X$; \emph{ii.} $X$ is complete multipartite. Whence, every induced subgraph of $X$ is complete multipartite also. An induced odd hole or paw in $X$ gives a contradiction, since neither of these graphs is complete multipartite. 
\end{proof}

\subsection{Auxiliary results}

We list in this section results that will be used throughout the proof of Theorem~\ref{thm:main_intro}.
Proposition~\ref{prop:strongperfthm}, the strong perfect graph theorem, will be of particular utility.
The sequel follows directly from definitions.

\begin{lem}\label{lem:wmp_comp_perf}
    The product graph $G \wprod H$ is perfect if and only if $\overline{G} \wprod \overline{H}$ is perfect.
\end{lem}

A \emph{hole} is an induced cycle. An \emph{antihole} is an induced co-cycle. An \emph{odd} (anti)hole has an odd number of vertices.

\begin{prop}[Strong Perfect Graph Theorem~{\cite{Chudnovsky2006}}]\label{prop:strongperfthm}
    A graph $G$ is perfect if and only if it has no odd holes or odd antiholes.
\end{prop}

\begin{cor}\label{cor:SPGT_cor}
    A graph $G$ is perfect if and only if both $G$ and $\overline{G}$ have no odd holes.
\end{cor}

The sequel follows as a corollary of the strong perfect graph theorem, but was originally proved by Lov\'asz in 1972~\cite{Lovasz1972a}.

\begin{prop}[Weak Perfect Graph Theorem]\label{prop:weakperfthm}
    A graph $G$ is perfect if and only if $\overline{G}$ is perfect.
\end{prop}

\begin{prop}[Cameron, Edmonds and Lov\'{a}sz{~\cite[Theorem $1'\,$]{Cameron1986}}]\label{prop:Cameron}
    Let $G_1$ and $G_2$ be perfect graphs and $G:=G_1\cup G_2$ be their union with $V(G_1) = V(G_2) = V(G)$ .
    Suppose that for any $x, x', x'' \in V(G)$, $\{ x,x' \}\in E(G_1)$ and $\{ x', x'' \} \in E(G_2)$ implies that $\{ x, x'' \}\in E(G)$.
    Then, $G$ is perfect.
\end{prop}

\begin{prop}[Ravindra, Parthasarathy{~\cite[Theorem 3.2]{Ravindra1977}}]\label{prop:Ravindra}
    The graph $G_1 \otimes G_2$ is perfect if and only if either
    \begin{enumerate}
        \item $G_1$ or $G_2$ is bipartite, or
        \item both $G_1$ and $G_2$ are $(\text{odd hole, paw})$-free.
    \end{enumerate}
\end{prop}

\begin{cor}\label{cor:GprodKn}
    The graph $G \wprod K_n$ is perfect if and only if either: $n=1$, $n=2$, or $G$ is $(\text{odd hole, paw})$-free. 
\end{cor}
\begin{proof}
    From definitions, $G \wprod K_n = G \otimes K_n$.
    For $n=1$ and $n=2$, $K_n$ is bipartite. For $n \geq 3$, observe that $K_n$ is $(\text{odd hole, paw})$-free.
\end{proof}

\subsubsection{Imperfect weak modular products}

For the proof of Theorem~\ref{thm:main_intro}, we will need to enumerate many pairs of graphs whose weak modular product is not perfect.
The main proof technique used is to find an offending odd hole or antihole in a given product graph. By the strong perfect graph theorem (Proposition~\ref{prop:strongperfthm}), the product is thus not perfect.
The following observation drastically reduces the work required. 

\begin{obs}
    Suppose $X$ is an induced subgraph of $G$ and $Y$ is an induced subgraph of $H$. Then, if $X \wprod Y$ is not perfect, $G \wprod H$ is not perfect.
    Conversely, if $G \wprod H$ is perfect, then $X \wprod Y$ is perfect.
\end{obs}

Thus if we have graph families $\Gamma_X$ and $\Gamma_Y$ such that every $G \in \Gamma_X$ contains an induced $X$ and every $H \in \Gamma_Y$ contains an induced $Y$, where $X \wprod Y$ is not perfect, every pair in $\Gamma_X \times \Gamma_Y$ has an imperfect weak modular product. 

For the upcoming results, we require the notion of an \emph{augment} of a graph.
\begin{defn}[Augment of a graph]
    Let $G$ be a graph on $n$ vertices. A graph $G'$ is an \emph{augment} of $G$ if $\abs{V(G')} = n+1$ and $G$ is an induced subgraph of $G'$.
\end{defn}

\begin{lem}\label{lem:C5trifree_augments_C5_P3}
    Let $G$ be a triangle-free augment of $C_5$. Then, $G \wprod P_3$ is not perfect.
\end{lem}
\begin{proof}
    Figure~\ref{fig:trifree_augments_C5_P3} shows the relevant graph products with induced odd holes and antiholes. The lemma follows from the strong perfect graph theorem.
    \begin{figure}[H]
        \centering
        \includegraphics[width=0.75\textwidth]{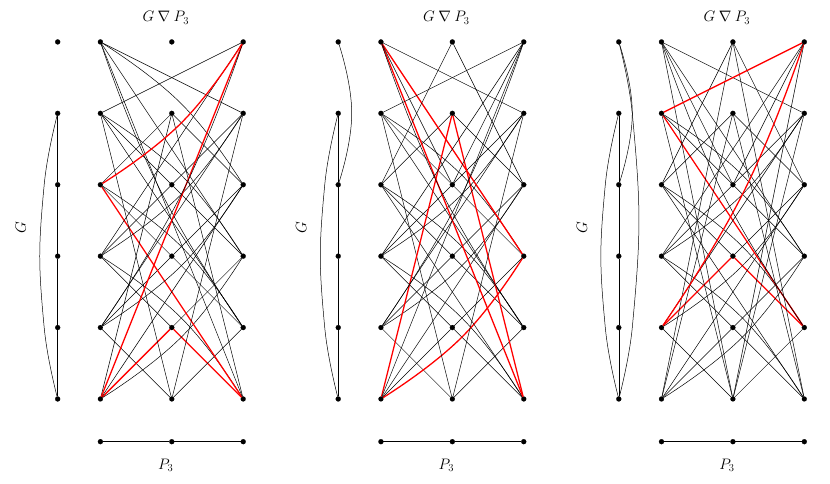}
        \caption{Weak modular product of $G$ and $P_3$, where $G$ is a triangle-free augment of $C_5$. An induced odd hole/antihole is denoted by a red, thick line.}
        \label{fig:trifree_augments_C5_P3}
    \end{figure}
\end{proof}

\begin{lem}\label{lem:C5trifree_augments_C5_K2uE1}
    Let $G$ be a triangle-free augment of $C_5$. Then, $G \wprod (K_2 \uplus E_1)$ is not perfect.
\end{lem}
\begin{proof}
    Figure~\ref{fig:trifree_augments_C5_K2uE1} shows the relevant graph products with induced odd holes and antiholes. The lemma follows from the strong perfect graph theorem.
    \begin{figure}[H]
        \centering
        \includegraphics[width=0.75\textwidth]{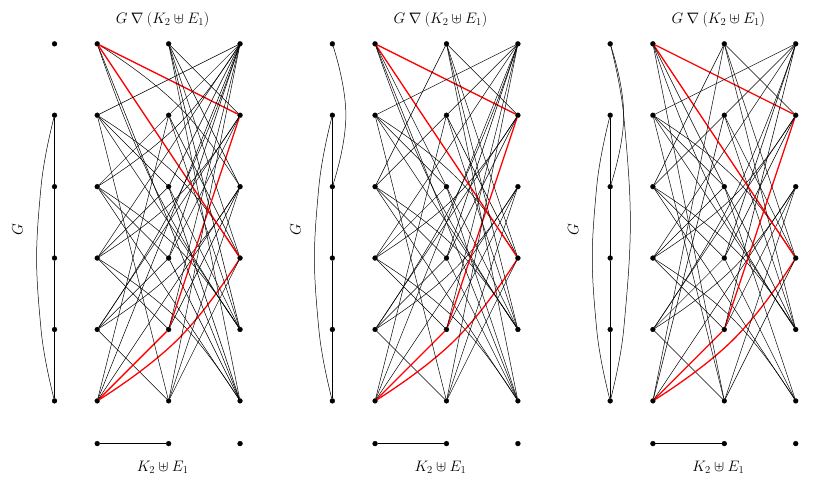}
        \caption{Weak modular product of $G$ and $K_2 \uplus E_1$, where $G$ is a triangle-free augment of $C_5$. An induced odd hole/antihole is denoted by a red, thick line.}
        \label{fig:trifree_augments_C5_K2uE1}
    \end{figure}
\end{proof}

\begin{lem}\label{lem:Bipartite_augments_P4_P3}
    Let $G$ be a bipartite augment of $P_4$. Then, $G \wprod P_3$ is not perfect.
\end{lem}
\begin{proof}
    Figure~\ref{fig:Bipartite_augments_P4_P3} shows the relevant graph products with induced odd holes and antiholes. The lemma follows from the strong perfect graph theorem.
    \begin{figure}[H]
        \centering
        \includegraphics[width=\textwidth]{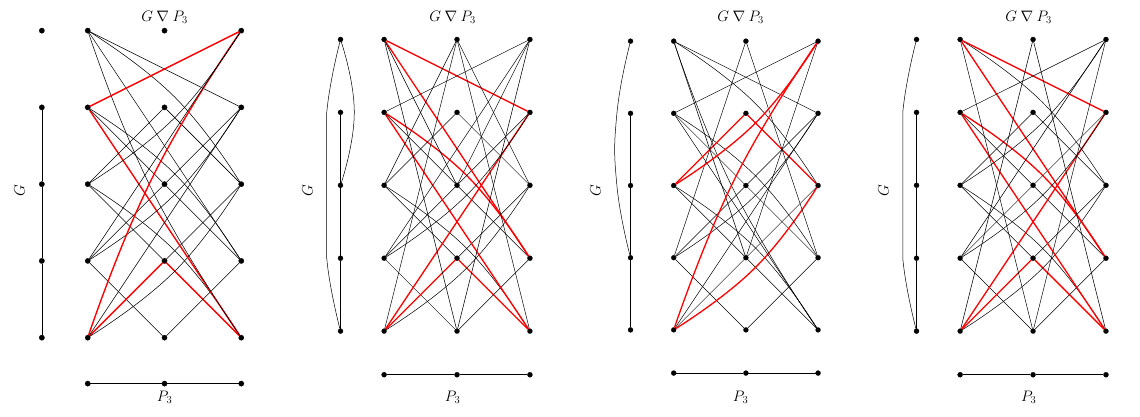}
        \caption{Weak modular product of $G$ and $P_3$, where $G$ is a bipartite augment of $P_4$. An induced odd hole/antihole is denoted by a red, thick line.}
        \label{fig:Bipartite_augments_P4_P3}
    \end{figure}
\end{proof}

\begin{lem}\label{lem:Bipartite_augments_P4_K2uE1}
    Let $G$ be a bipartite augment of $P_4$. Then, $G \wprod (K_2 \uplus E_1)$ is not perfect.
\end{lem}
\begin{proof}
    Figure~\ref{fig:Bipartite_augments_P4_K2uE1} shows the relevant graph products with induced odd holes and antiholes. The lemma follows from the strong perfect graph theorem.
    \begin{figure}[H]
        \centering
        \includegraphics[width=\textwidth]{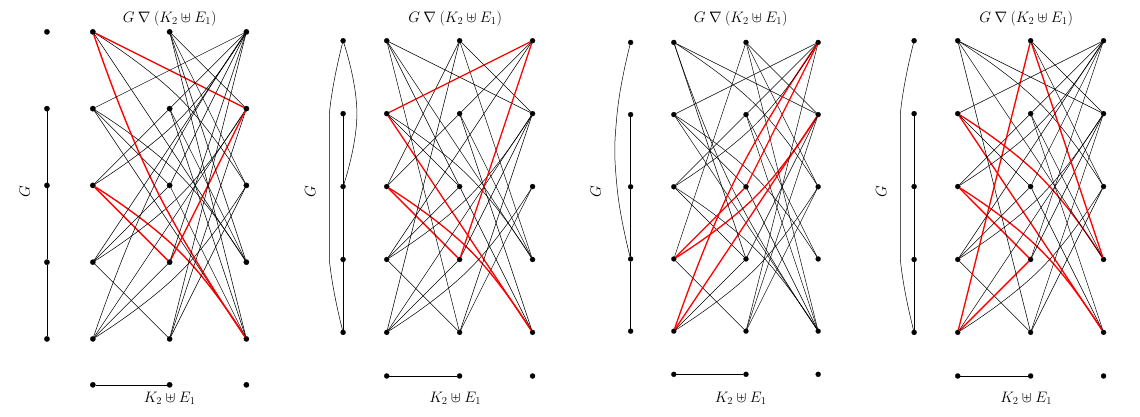}
        \caption{Weak modular product of $G$ and $K_2 \uplus E_1$, where $G$ is a bipartite augment of $P_4$. An induced odd hole/antihole is denoted by a red, thick line.}
        \label{fig:Bipartite_augments_P4_K2uE1}
    \end{figure}
\end{proof}

\begin{lem}\label{lem:P3_cricket_dart_hourglass}
    Let $G \in \{ \text{cricket, dart, hourglass} \}$. Then, $G \wprod P_3$ is not perfect.
\end{lem}
\begin{proof}
    Figure~\ref{fig:P3_cricket_dart_hourglass} shows the relevant graph products with induced odd holes and antiholes. The lemma follows from the strong perfect graph theorem.
    \begin{figure}[H]
        \centering
        \includegraphics[width=0.7\textwidth]{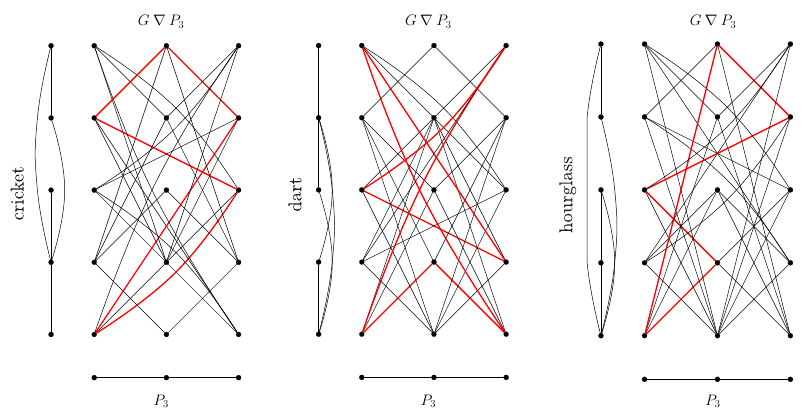}
        \caption{Weak modular product of $G$ and $P_3$, where $G \in \{ \text{cricket, dart, hourglass} \}$. An induced odd hole/antihole is denoted by a red, thick line.}
        \label{fig:P3_cricket_dart_hourglass}
    \end{figure}
\end{proof}

\begin{lem}\label{lem:K2uE1_products_nonperf}
    Let $G \in \{ K_{1,1,2}, Y, P_4 \uplus E_1, K_{2,2} \uplus E_1, P_5 \}$. Then $(K_2 \uplus E_1) \wprod G$ is not perfect.
\end{lem}
\begin{proof}
    Figure~\ref{fig:K2uE1_prods} shows the relevant graph products with induced odd holes and antiholes. The lemma follows from the strong perfect graph theorem.
    \begin{figure}[H]
        \centering
        \includegraphics[width=\textwidth]{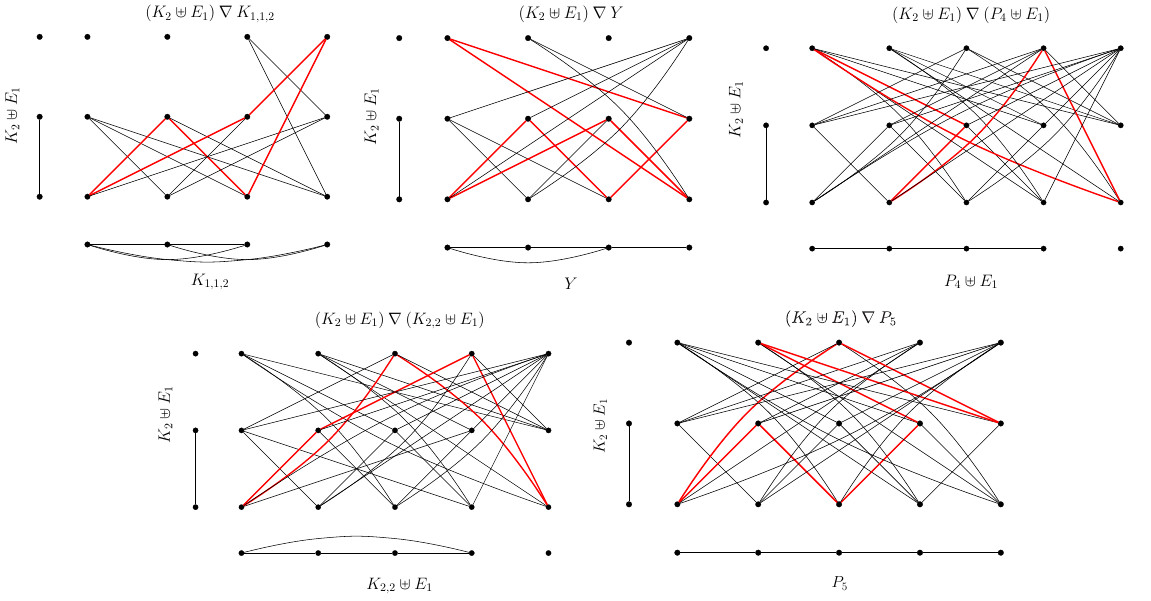}
        \caption{Weak modular product of $G$ and $K_2 \uplus E_1$, where $G \in \{ K_{1,1,2}, Y, P_4 \uplus E_1, K_{2,2} \uplus E_1, P_5 \}$. An induced odd hole/antihole is denoted by a red, thick line.}
        \label{fig:K2uE1_prods}
    \end{figure}
\end{proof}

\begin{lem}\label{lem:P3_prods}
    Let $G \in \{ K_2 \uplus E_2, P_3 \uplus E_1, P_5, K_{1,1,2} \uplus E_1, 3K_2 \}$. Then $P_3 \wprod G$ is not perfect.
\end{lem}
\begin{proof}
    Figure~\ref{fig:P3_prods} shows the relevant graph products with induced odd holes and antiholes. The lemma follows from the strong perfect graph theorem.
    \begin{figure}[H]
        \centering
        \includegraphics[width=\textwidth]{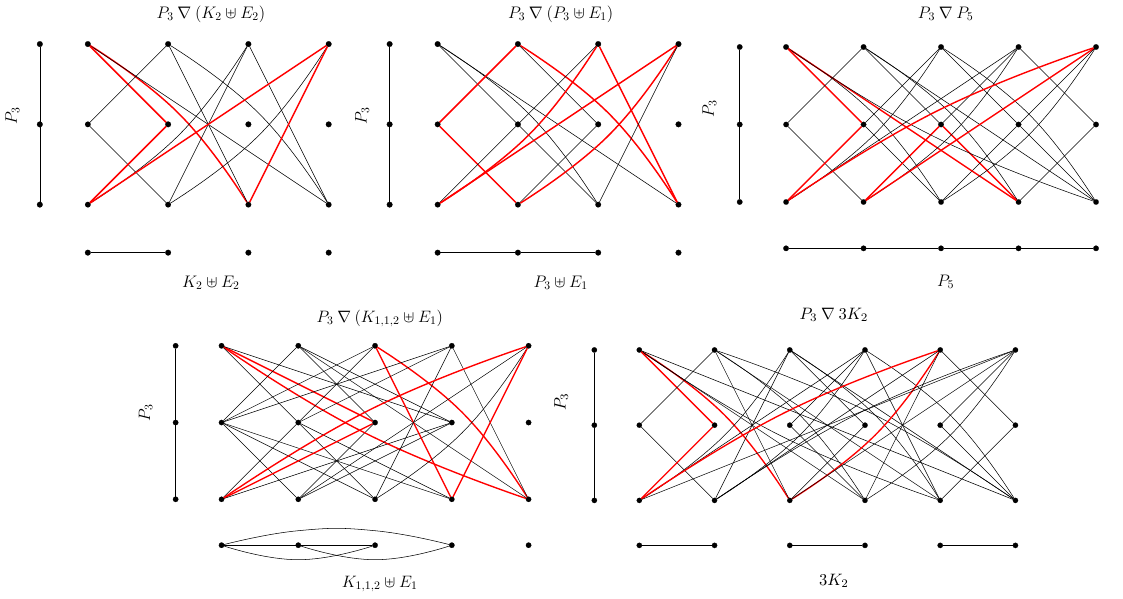}
        \caption{Weak modular product of $G$ and $P_3$, where $G \in \{ K_2 \uplus E_2, P_3 \uplus E_1, P_5, K_{1,1,2} \uplus E_1, 3K_2 \}$. An induced odd hole/antihole is denoted by a red, thick line.}
        \label{fig:P3_prods}
    \end{figure}
\end{proof}

\begin{lem}\label{lem:P4_prods}
    Let $G \in \{ 2K_2, Y, K_{2,2}, K_{1,1,2} \}$. Then $P_4 \wprod G$ is not perfect.
\end{lem}
\begin{proof}
    Figure~\ref{fig:P4_prods} shows the relevant graph products with induced odd holes and antiholes. The lemma follows from the strong perfect graph theorem.
    \begin{figure}[H]
        \centering
        \includegraphics[width=0.6\textwidth]{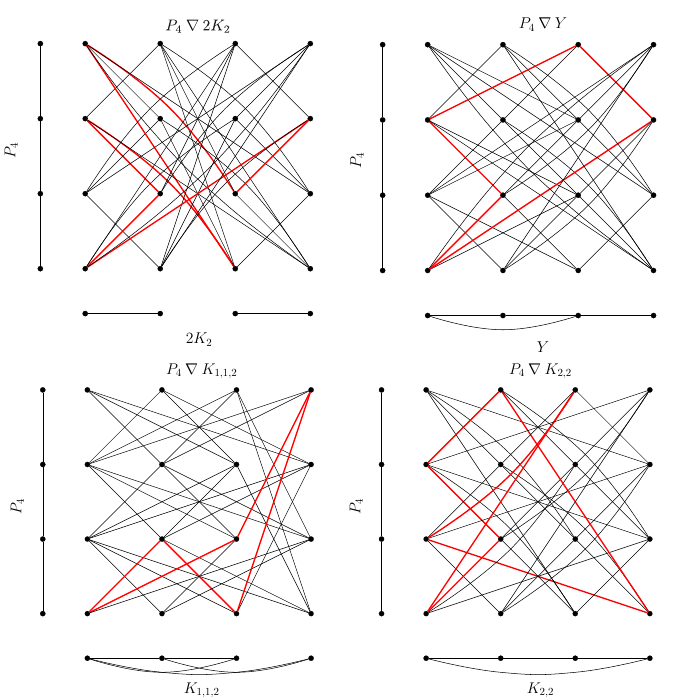}
        \caption{Weak modular product of $G$ and $P_4$, where $G \in \{ 2K_2, Y, K_{2,2}, K_{1,1,2}\}$. An induced odd hole/antihole is denoted by a red, thick line.}
        \label{fig:P4_prods}
    \end{figure}
\end{proof}

\begin{lem}\label{lem:K22_prods}
    Let $G \in \{ P_3 \uplus E_1, K_2 \uplus E_2\}$. Then $K_{2,2} \wprod G$ is not perfect.
\end{lem}
\begin{proof}
    Figure~\ref{fig:K22_prods} shows the relevant graph products with induced odd holes and antiholes. The lemma follows from the strong perfect graph theorem.
    \begin{figure}[H]
        \centering
        \includegraphics[width=0.65\textwidth]{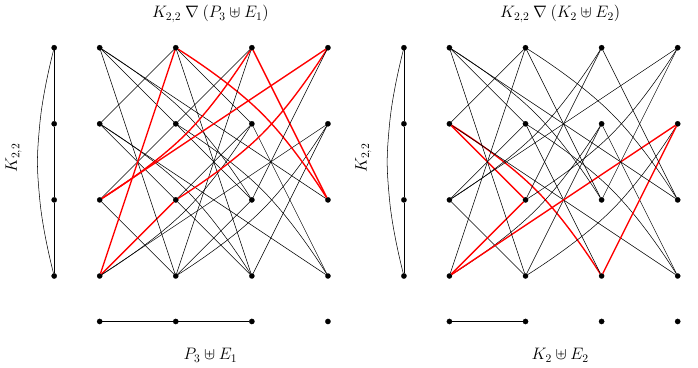}
        \caption{Weak modular product of $G$ and $K_{2,2}$, where $G \in \{ P_3 \uplus E_1, K_2 \uplus E_2 \}$. An induced odd hole/antihole is denoted by a red, thick line.}
        \label{fig:K22_prods}
    \end{figure}
\end{proof}

\begin{lem}\label{lem:C5_prods}
    Let $G \in \{ K_3, 2K_2, K_{1,3}, K_{2,2} \}$. Then $C_5 \wprod G$ is not perfect.
\end{lem}
\begin{proof}
    Figure~\ref{fig:C5_prods} shows the relevant graph products with induced odd holes and antiholes. The lemma follows from the strong perfect graph theorem.
    \begin{figure}[H]
        \centering
        \includegraphics[width=\textwidth]{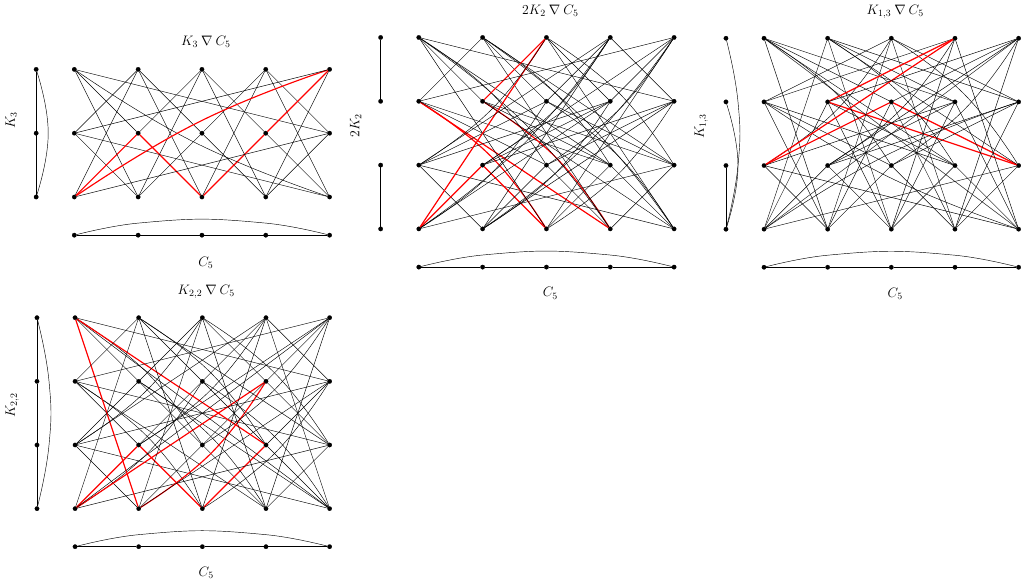}
        \caption{Weak modular product of $G$ and $C_5$, where $G \in \{ K_3, 2K_2, K_{1,3}, K_{2,2} \}$. An induced odd hole/antihole is denoted by a red, thick line.}
        \label{fig:C5_prods}
    \end{figure}
\end{proof}

\subsubsection{Perfect weak modular products}

In this section we find perfect weak modular product graphs, using tools from previous sections, notably Proposition~\ref{prop:Cameron}.

For the next lemma, we require the concept of a \emph{line graph}.
For a graph $G$, its line graph $L(G)$ is the graph where $V(L(G)) = E(G)$ and $\{e_1,e_2\} \in E(L(G))$ if and only if the edges $e_1, e_2 \in E(G)$ share a vertex in $V(G)$.

\begin{lem}\label{lem:C5perf_prods}
    Suppose $G$ is a graph containing an induced $P_3$. Then, $G \wprod C_5$ is perfect if and only if $G \in \{ P_3, P_4, C_5 \}$.
\end{lem}
\begin{proof}
    $(\Leftarrow)$ As demonstrated by Figure~\ref{fig:mantaray} and Table~\ref{tab:line_graph_v}, $C_5 \wprod C_5$ is the line graph $L(M)$ of the graph $M$, where $M$ is defined in Figure~\ref{fig:mantaray}.
    Observe further that $M$ is bipartite, with partite sets $\{ 0, 3, 5, 7, 8 \}$ and $\{ 1, 2, 4, 6, 9 \}$ using the labelling of Figure~\ref{fig:mantaray}.
    It is a well-known result that line graphs of bipartite graphs are perfect, and so $C_5 \wprod C_5$ is perfect.
    Moreover, since $P_3$ and $P_4$ are induced subgraphs of $C_5$, $P_3 \wprod C_5$ and $P_4 \wprod C_5$ are perfect.
    
    $(\Rightarrow)$ We prove the contrapositive. First observe that $C_5 \wprod K_3$ is not perfect, by Proposition~\ref{prop:Ravindra}.
    Let $G$ be triangle-free and suppose $G$ is not bipartite. Then, by Lemma~\ref{lem:triangle_free_odd_hole} $G$ has an induced odd hole, in which case $G$ contains an induced $C_5$ or an induced $P_5$.
    In the former case, $G \wprod C_5$ is perfect only if $G \cong C_5$, for if $G$ contains a triangle-free augment of $C_5$, then $G \wprod C_5$ is not perfect by Lemma~\ref{lem:C5trifree_augments_C5_P3}.
    In the latter case, from Lemma~\ref{lem:P3_prods} $P_4 \wprod P_5$ is not perfect and so $C_5 \wprod G$ is not perfect.
    Now suppose $G$ is bipartite. Then, by Lemma~\ref{lem:compbipP4} $G$ either has an induced $P_4$ or is a disjoint union of complete bipartites.
    In the former case we see $G \wprod C_5$ is perfect only if $G \wprod P_4$, for if $G$ contains a bipartite augment of $P_4$, then $G \wprod P_4$ is not perfect by Lemma~\ref{lem:Bipartite_augments_P4_P3}.
    In the latter case, $G$ either contains an induced $K_{2,2}$ or is a disjoint union of stars.
    If $G$ contains an induced $K_{2,2}$, $G \wprod C_5$ is not perfect by Lemma~\ref{lem:C5_prods}.
    If $G$ is a disjoint union of stars and $G \not\cong P_3$, $G$ contains either an induced $P_3 \uplus E_1$ or an induced $K_{1,3}$. Then, $G\wprod C_5$ is not perfect by Lemmas~\ref{lem:P3_prods} and \ref{lem:C5_prods} respectively. 
\end{proof}

    \begin{table}[H]
        \centering
        \begin{tabular}{@{}c|lllll@{}}
        \toprule
                          &  $(\,\cdot\, , 1)$ & $(\,\cdot\, , 2)$ & $(\,\cdot\, , 3)$ & $(\,\cdot\, , 4)$ & $(\,\cdot\, , 5)$ \\ \midrule
        $(1, \,\cdot \,)$ & $\{ 0, 1 \}$    & $\{ 5, 6 \}$    & $\{ 2, 8 \}$    & $\{ 3, 9 \}$    & $\{ 4, 7 \}$    \\
        $(2, \,\cdot \,)$ & $\{ 6, 7 \}$    & $\{ 0, 2 \}$    & $\{ 5, 9 \}$    & $\{ 4, 8 \}$    & $\{ 1, 3 \}$    \\
        $(3, \,\cdot \,)$ & $\{ 2, 3 \}$    & $\{ 7, 9 \}$    & $\{ 0, 4 \}$    & $\{ 1, 5 \}$    & $\{ 6, 8 \}$    \\
        $(4, \,\cdot \,)$ & $\{ 8, 9 \}$    & $\{ 3, 4 \}$    & $\{ 1, 7 \}$    & $\{ 0, 6 \}$    & $\{ 2, 5 \}$    \\
        $(5, \,\cdot \,)$ & $\{ 4, 5 \}$    & $\{ 1, 8 \}$    & $\{ 3, 6 \}$    & $\{ 2, 7 \}$    & $\{ 0, 9 \}$    \\ \bottomrule
        \end{tabular}
        \caption{Assignment of vertices in $C_5 \wprod C_5$ to edges in $M$, with vertex labels imposed in Figure~\ref{fig:mantaray}. Two vertices in $C_5 \wprod C_5$ are adjacent if and only if the corresponding edges in $M$ per the table share a common vertex. Thus $C_5 \wprod C_5 \cong L(M)$, where $L(M)$ is the line graph of $M$.}
        \label{tab:line_graph_v}
    \end{table}
    
    \begin{figure}[H]
        \centering
        \includegraphics[width=0.7\textwidth]{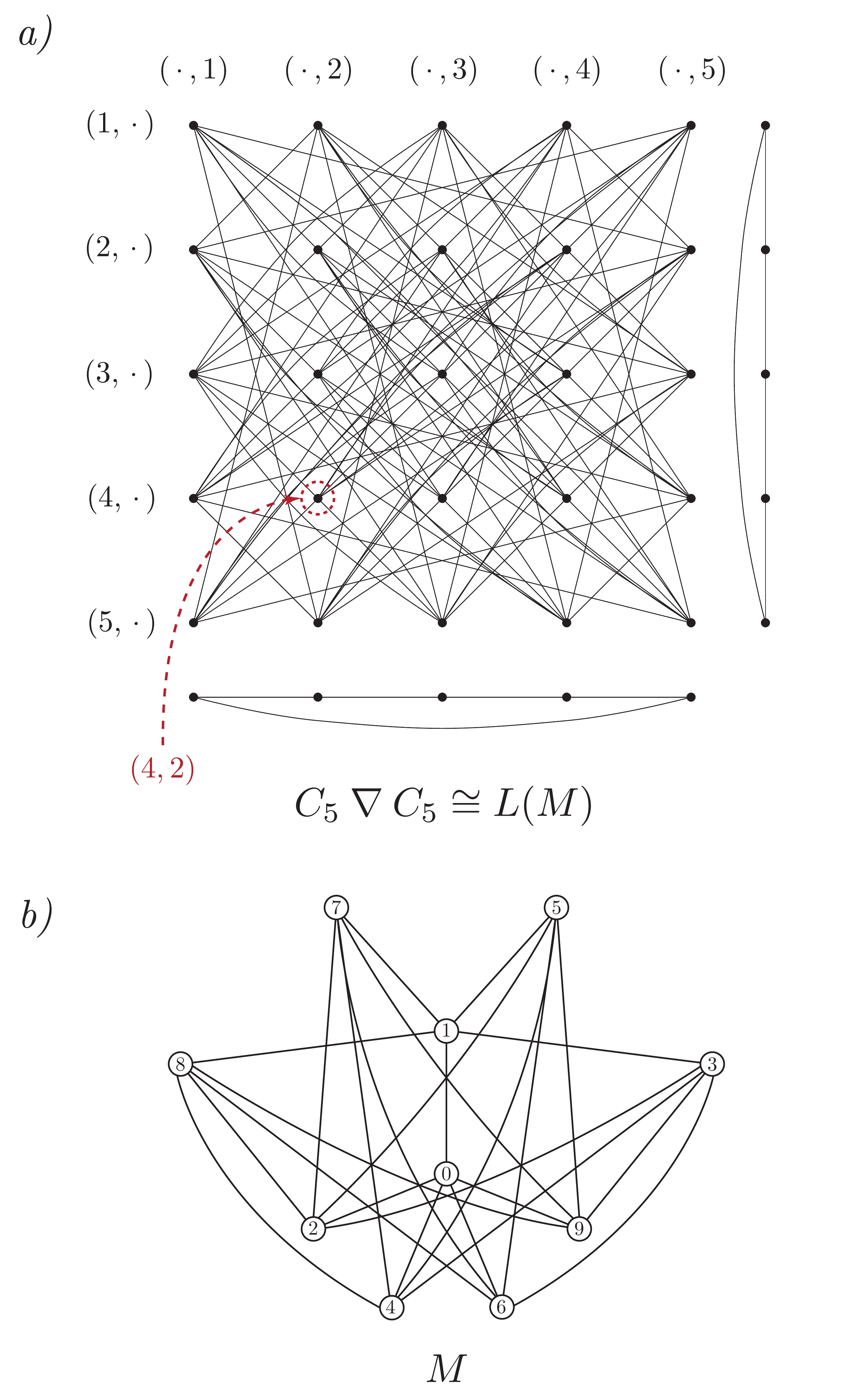}
        \caption{\emph{a)} The graph $C_5 \wprod C_5$ with vertex labelling indicated. \emph{b)} The graph we denote by $M$, with vertex labels. The graph $C_5 \wprod C_5$ is isomorphic to the line graph of $M$, i.e. $C_5 \wprod C_5 \cong L(M)$. The assignment of vertices in $C_5 \wprod C_5$ to edges in $M$ is given by Table~\ref{tab:line_graph_v}.}
        \label{fig:mantaray}
    \end{figure}

It is interesting to note that the graph $C_5 * C_5$ is not perfect, where $*$ is the strong, direct, Cartesian or associative product~\cite{Ravindra1977,Ravindra1978}.

\begin{cor}\label{cor:C5prodK2uE1}
    The graph $C_5 \wprod (K_2 \uplus E_1)$ is perfect.
\end{cor}
\begin{proof}
    The graph $C_5 \wprod P_3$ is perfect taken with Lemma~\ref{lem:wmp_comp_perf}.
\end{proof}

\begin{lem}\label{lem:GHdisj_cliqs}
    Let $G$ and $H$ be the disjoint union of cliques. Then, $G \wprod H$ is perfect.
\end{lem}
\begin{proof}
    We shall proceed by using Proposition~\ref{prop:Cameron}, taking $G_1$ in the theorem statement as $G \otimes H$ and correspondingly, $G_2$ as $\overline{G} \otimes \overline{H}$. The assumptions of the theorem are satisfied, namely, $G_1$ and $G_2$ are perfect. This follows from Proposition~\ref{prop:Ravindra}, since $G$, $\overline{G}$, $H$ and $\overline{H}$ are $(\text{odd hole, paw})$-free.

    Now call $G = \biguplus_i^k K_{r_i}$ and $H = \biguplus^\ell_j K_{s_j}$. 
    It remains to show that $(x,y) \sim_{G\otimes H} (x',y')$ and $(x',y') \sim_{\overline{G} \otimes \overline{H}} (x'', y'')$ implies $(x,y) \sim_{G\otimes H \cup \overline{G} \otimes \overline{H}} (x'',y'') \equiv (x,y) \sim_{G \wprod H} (x'',y'')$, where $x, x', x'' \in V(G)$ and $y, y', y'' \in V(H)$. 
    Now, denote the subsets of vertices comprising the cliques in $G$ by $U_i$ respectively, that is, the $i$\supth clique of $G$ is induced on the vertex set $U_i$. The $j$\supth clique of $H$, $K_{n_j}$, is induced on the vertex set $V_j$.

    From definitions, we have that $(x,y) \sim_{G\otimes H} (x',y')$ if and only if $x \in U_i$, $x'\in U_i$, $x \neq x'$ for some $i \in [k]$, and $y \in V_j$, $y'\in V_j$, $y \neq y'$ for some $j \in [\ell]$.
    Moreover, we have that $(x,y) \sim_{\overline{G} \otimes \overline{H}} (x', y')$ if and only if $x \in U_i$, $x' \in U_{i'}$ for $i \neq i'$, $i,i'\in [k]$ and $y \in V_j$, $y' \in V_{j'}$ for $j \neq j'$, $j,j'\in [\ell]$.

    Suppose the edge $\{(x,y), (x',y')\}$ exists in $G\otimes H$ and $\{(x',y'), (x'',y'')\}$ exists in $\overline{G} \otimes \overline{H}$.
    We have that $x \in U_i$, $x' \in U_i$ for $x \neq x'$, $x'' \in U_{\tilde{i}}$ for $i \neq \tilde{i}$ and $y \in V_j$, $y' \in V_j$ for $y \neq y'$, $y'' \in V_{\tilde{j}}$ for $j \neq \tilde{j}$.
    Clearly, we have that $(x,y) \sim_{\overline{G} \otimes \overline{H}} (x'', y'')$, giving that $(x,y) \sim_{G\otimes H \cup \overline{G} \otimes \overline{H}} (x'',y'')$, and so $G \wprod H$ is perfect by Proposition~\ref{prop:Cameron}.
\end{proof}
\begin{cor}\label{cor:completemultiprodperf}
    Let $G$ and $H$ be complete multipartite graphs. Then, $G \wprod H$ is perfect.
\end{cor}
\begin{proof}
    Take complements and use Lemmas~\ref{lem:wmp_comp_perf}~and~\ref{lem:disjoint_compl_multi}.
\end{proof}

\begin{lem}\label{lem:KruKs_Kmn}
    The graph $(K_r \uplus K_s) \wprod K_{m,n}$ is perfect.
\end{lem}
\begin{proof}
    Let $G = (K_r \uplus K_s)$ and $H = K_{m,n}$.
    Call the vertices comprising the cliques in $G$ $U_0$ and $U_1$ respectively, such that $V(G) = U_0 \cup U_1$. 
    Likewise, for brevity call the vertices comprising the partite sets in $H$ $V_0$ and $V_1$ respectively, so that $V(H) = V_0 \cup V_1$.
    From the definition of the weak modular product $G \wprod H$, we have that a vertex $(x,y) \in U_{z_1} \times V_{z_2}$ is adjacent to a vertex in $(x', y') \in U_{z_3} \times V_{z_4}$ if and only if $x \neq x'$, $y \neq y'$ and either: $z_1 = z_3$ and $z_2 \neq z_4$; or $z_1 \neq z_3$ and $z_2 = z_4$ for $z_1, z_2, z_3, z_4 \in \{0, 1\}$. There are no other edges.
    Notice that $G\wprod H$ comprises a complete bipartite graph with partite sets $U_0 \times V_0 \cup U_1 \times V_1$ and $U_0 \times V_1 \cup U_1 \times V_0$, with a perfect matching removed.
    Since $G \wprod H$ is bipartite, it is perfect.
\end{proof}

We now require a number of auxiliary lemmas to prove Proposition~\ref{prop:twocliq_disjstrscliqs}.

\begin{lem}\label{lem:Kr_uplus_Ks_wprod_G_uplus_H}
    Suppose $(K_r \uplus K_s) \wprod (G \uplus K_1)$ and $(K_r \uplus K_s) \wprod (H \uplus K_1)$ are perfect. Then, $(K_r \uplus K_s) \wprod (G \uplus H)$ is perfect.
\end{lem}
\begin{proof}
    For brevity, we denote $(K_r \uplus K_s) \wprod (G \uplus H)$ by $\Lambda$ and $V(\Lambda)$ by $U$.
    We draw the structure of $\Lambda$ in Figure~\ref{fig:Kr_uplus_Ks_wprod_G_uplus_H}
    Observe that $U$ is partitioned into four disjoint subsets of vertices: $V(K_r) \times V(G)$, $V(K_r) \times V(H)$, $V(K_s) \times V(G)$ and $V(K_s) \times V(H)$, which we respectively denote $U_1$, $U_2$, $U_3$ and $U_4$.
    Now for the sake of contradiction suppose $\Lambda$ is not perfect. Then, by the strong perfect graph theorem it contains an induced odd hole or antihole, which we call $X$.
    The vertices of $X$, $V(X)$, cannot lie solely in one partitioned subset of $U$, for in this case $X$ is an induced subgraph of a perfect graph and we have a contradiction.

    Suppose now $X$ has vertices in two of the partitioned subsets, \emph{i.e.} $V(X) \subseteq U_i \cup U_j$ for $i,j \in \{ 1,2,3,4 \}$, $i \neq j$ and $V(X) \not\subseteq U_i$, $V(X) \not\subseteq U_j$.
    If $V(X) \subseteq U_1 \cup U_2$ or $V(X) \subseteq U_3 \cup U_4$, then $X$ is an induced subgraph of a disjoint union of perfect graphs and we have a contradiction.
    If $V(X) \subseteq U_1 \cup U_3$ or $V(X) \subseteq U_2 \cup U_4$ then $X$ is an induced subgraph of $(K_r \uplus K_s) \wprod G$ or $(K_r \uplus K_s) \wprod H$ respectively, which are perfect by assumption, yielding a contradiction.
    We also obtain a contradiction when $V(X) \subseteq U_1 \cup U_4$ or $V(X) \subseteq U_2 \cup U_3$, as $\overline{\Lambda}[U_1 \cup U_4]$ and $\overline{\Lambda}[U_2 \cup U_3]$ are disjoint unions of perfect graphs and so $\Lambda[U_1 \cup U_4]$, $\Lambda[U_2 \cup U_3]$ are perfect by the weak perfect graph theorem.

    Assume now that $X$ has vertices lying in three of the partitioned subsets, \emph{i.e.} $V(X)\cap U_i = \emptyset$ for exactly one $i \in \{1,2,3,4\}$. Furthermore, let $i = 4$ without loss of generality, so that $V(X) \subseteq U_1 \cup U_2 \cup U_3$.
    Observe that every vertex of $U_2$ is adjacent to every vertex of $U_3$ in $\Lambda$ by definition, and vice versa.
    Moreover, $U_2$ is adjacent to every vertex of $U_1$ in $\overline{\Lambda}$ and vice versa.
    Suppose now that at least two vertices of $X$ lie in $U_2$. By assumption there is at least one vertex of $X$ in each of $U_3$ and $U_1$. Thus, $\Lambda[V(X)]$ and $\overline{\Lambda}[V(X)]$ both contain a triangle and we have a contradiction with Corollary~\ref{cor:SPGT_cor}, as neither $X$ nor $X$ are an odd hole.
    Now suppose one vertex of $X$ lies in $U_2$. Thus, $X$ is an induced subgraph of $(K_r \uplus K_s) \wprod (G \uplus K_1)$ and we have a contradiction since $(K_r \uplus K_s) \wprod (G \uplus K_1)$ is perfect by assumption.

    Finally, assume there is a vertex from $X$ in every partition of $U$, \emph{i.e.} $V(X)\cap U_i \neq \emptyset$ for every $i \in \{1,2,3,4\}$.
    By the pigeonhole principle, there is at least one $j \in \{ 1, 2, 3, 4 \}$ such that $U_j$ has two or more vertices from $X$. Moreover, for any $j$ there exist $j', j'' \in \{1,2,3,4\}$ where $j \neq j'$, $j \neq j''$, $j' \neq j''$ such that every vertex in $U_j$ is connected to every vertex in $U_{j'}$ in $\Lambda$ and every vertex in $U_j$ is connected to every vertex in $U_{j''}$ in $\overline{\Lambda}$. By the argument in the previous paragraph, this contradicts the assumption that $X$ is an odd hole or antihole.
\end{proof}
\begin{figure}[H]
    \centering
    \includegraphics[width=0.50\textwidth]{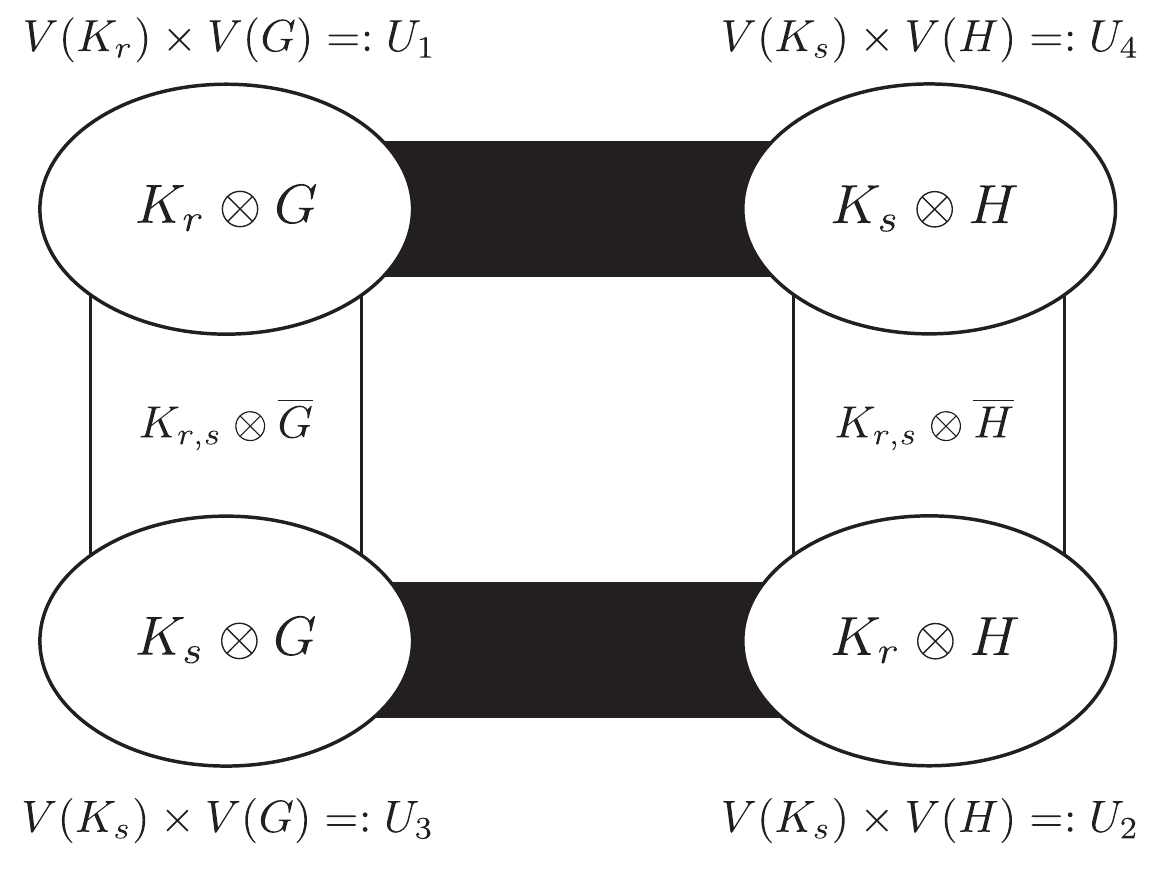}
    \caption{Structure of the graph $(K_r \uplus K_s) \wprod (G \uplus H)$ with the vertex set partition $(U_1, U_2, U_3, U_4)$ defined in the proof of Lemma~\ref{lem:Kr_uplus_Ks_wprod_G_uplus_H}. Every vertex in $U_1$ is adjacent to every vertex in $U_4$ and vice versa. Every vertex in $U_2$ is adjacent to every vertex in $U_3$ and vice versa.} 
    \label{fig:Kr_uplus_Ks_wprod_G_uplus_H}
\end{figure}

\begin{cor}\label{cor:twocliq_disj_G_H}
    Suppose $(K_r \uplus K_s) \wprod (G \uplus K_1)$ and $(K_r \uplus K_s) \wprod (H \uplus K_1)$ are perfect. Then, $(K_r \uplus K_s) \wprod (\biguplus_{i=1}^{k_1} G \uplus \biguplus_{i=1}^{k_2} H)$ is perfect for any $k_1, k_2 \in \mathbb{N}$.
\end{cor}

\begin{lem}\label{lem:twocliq_wprod_K1r_uplus_K1}
    The graph $(K_r \uplus K_s) \wprod (K_{1,m} \wprod K_1)$ is perfect.
\end{lem}
\begin{proof}
    For brevity we denote the graph $(K_r \uplus K_s) \wprod (K_{1,m} \uplus K_1)$ by $G$.
    Furthermore, we impose the vertex labelling of Figure~\ref{fig:Kr_uplus_Ks_wprod_K1m_uplus_K1}.
    We show that $G$ satisfies the definition of a perfect graph, namely that $\omega(X) = \chi(G)$ for all induced subgraphs $X$.
    First, observe that $\chi(G)$ is 3-colourable, according to the colouring in Figure~\ref{fig:Kr_uplus_Ks_wprod_K1m_uplus_K1}.
    Now, let $z\in \{0,1\}$ and $\bar{z}$ be its binary complement and let $X$ be an induced subgraph of $G$.
    Suppose $X$ includes a vertex from $U_z \times V_{1,2}$, a vertex from $U_z \times V_0$ and a vertex from $U_{\overline{z}} \times V_{1,1}$.
    Then $X$ contains a triangle.
    From inspection of Figure~\ref{fig:Kr_uplus_Ks_wprod_K1m_uplus_K1} one sees that $G$ is $K_4$-free, so $\omega(X)=3$ in this case.
    Moreover, since $G$ has a 3-colouring, $\omega(X) = \chi(X) = 3$ as $\omega(H) \leq \chi(H)$ for any graph $H$.
    If $X$ contains no three such vertices, we see from Figure~\ref{fig:Kr_uplus_Ks_wprod_K1m_uplus_K1} that $X$ is bipartite and so $\omega(X) = \chi(X) = 2$.   
\end{proof}
\begin{figure}[H]
    \centering
    \includegraphics[width=0.75\textwidth]{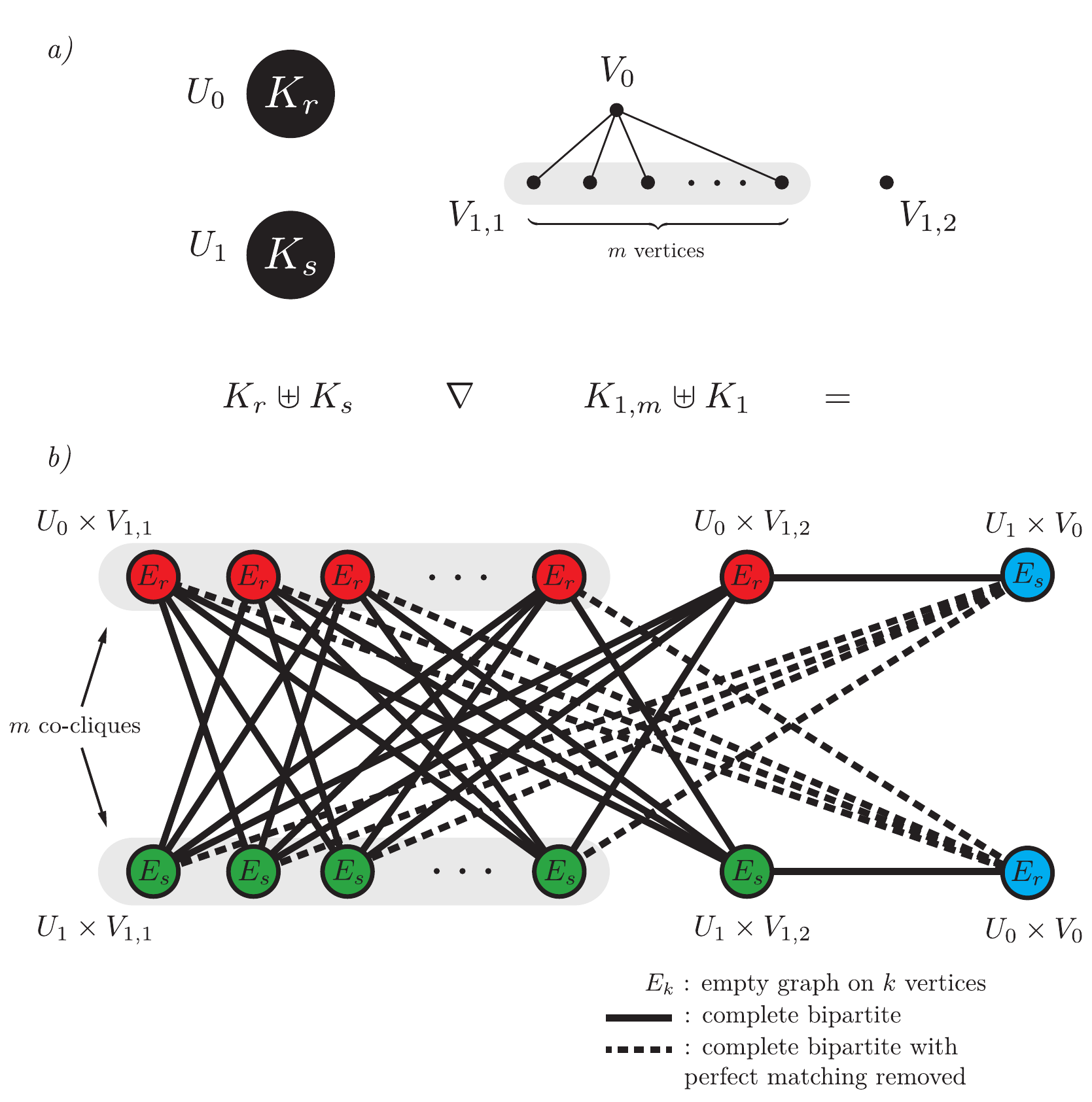}
    \caption{\emph{a)} Illustration of the graphs $K_r \uplus K_s$ and $K_{1, m} \uplus K_1$ with vertex labellings indicated.
    Illustration of the graph $(K_r \uplus K_s) \wprod (K_{1,m} \uplus K_1)$, with vertex labelling and a 3-colouring indicated.
    Each node represents an empty graph on either $r$ or $s$ vertices. A full line represents a fully bipartite graph induced over the end nodes and a dashed line represents a complete bipartite graph with a perfect matching removed.} 
    \label{fig:Kr_uplus_Ks_wprod_K1m_uplus_K1}
\end{figure} 

\begin{prop}\label{prop:twocliq_disjstrscliqs}
    Let $G = K_r \uplus K_s$ and let $H$ be a disjoint union of stars and cliques. Then, $G \wprod H$ is perfect.
\end{prop}
\begin{proof}
    Follows immediately from Corollary~\ref{cor:twocliq_disj_G_H}, Lemma~\ref{lem:twocliq_wprod_K1r_uplus_K1} and Lemma~\ref{lem:GHdisj_cliqs}.
\end{proof}
\begin{lem}
    \label{lem:P4wprodK1r}
    The graph $P_4 \wprod K_{1,r}$ is perfect for any $r \geq 1$.\end{lem}
\begin{proof}
    \begin{figure}[H]
        \centering
        \includegraphics[width=0.85\textwidth]{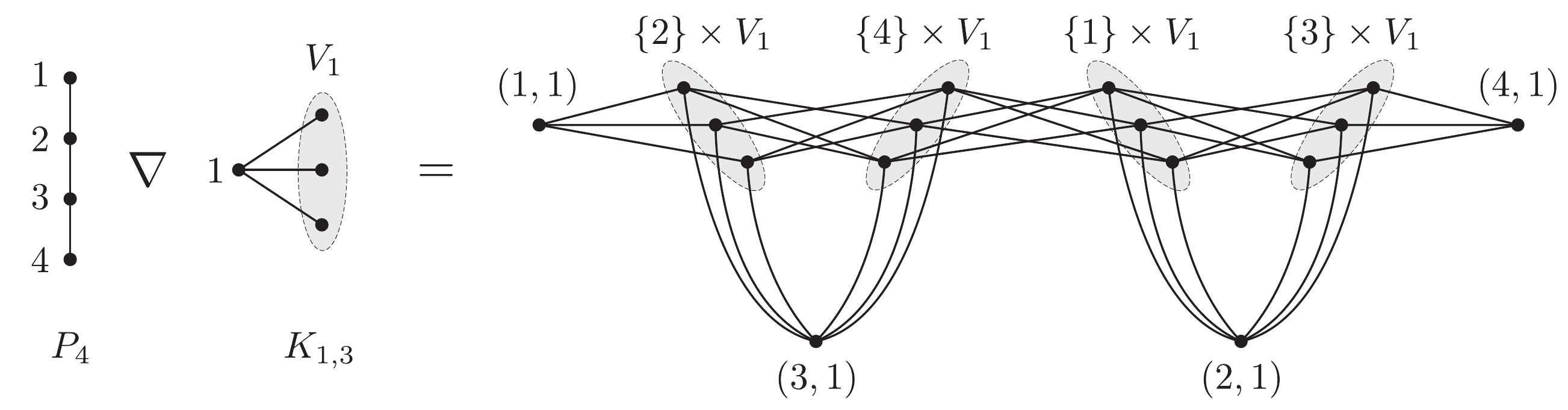}
        \caption{Illustration of the graph $P_4 \wprod K_{1,r}$, with $r=3$ for concreteness. The sets $\{i\} \times V_1$ for $i \in \{1,2,3,4\}$ each have cardinality $r$.
        The subgraphs of $P_4 \wprod K_{1,r}$ induced on the sets $\{2,4\} \times V_1$, $\{1,3\} \times V_1$ and $\{1,4\} \times V_1$ are complete bipartite graphs with a perfect matching removed.
        The vertex $(1,1)$ is connected to all vertices in $\{2\}\times V_1$.
        The vertex $(4,1)$ is connected to all vertices in $\{3\}\times V_1$.
        The vertex $(3,1)$ is connected to all vertices in $\{2\}\times V_1$ and $\{4\}\times V_1$.
        The vertex $(2,1)$ is connected to all vertices in $\{1\}\times V_1$ and $\{3\}\times V_1$.
        } 
        \label{fig:P4_wprod_K_1r}
    \end{figure}
    The case $r=1$ is trivial by Corollary~\ref{cor:GprodKn}. Impose now the vertex labelling from Figure~\ref{fig:P4_wprod_K_1r}. We have drawn $P_4 \wprod K_{1,r}$ for $r=3$ in Figure~\ref{fig:P4_wprod_K_1r}, the result readily generalises to any $r \geq 2$. We now use the strong perfect graph theorem. Observe that any induced subgraph of $P_4 \wprod K_{1,r}$ not including both vertices $(3,1)$ and $(2,1)$ is bipartite, so is perfect. Thus, any induced subgraph containing an odd hole or antihole must include either $(3,1)$ or $(2,1)$.
    Without loss of generality, consider an odd cycle $C = x_1,x_2,\ldots, x_{2k+2}$ (for $k\geq 2$) starting and beginning at vertex $(3,1)$, i.e. $x_1 = x_{2k+2} = (3,1)$.
    Clearly, either $x_2 \in \{2\}\times V_1$ and $x_{2k + 1} \in \{4\}\times V_1$ or $x_2 \in \{4\}\times V_1$ and $x_{2k + 1} \in \{2\}\times V_1$. Then, the vertices $\{x_1, x_2, x_{2k + 1}\}$ induce a triangle and so $(P_4 \wprod K_{1,r})[C]$ is not an odd hole.
    Thus, $P_4 \wprod K_{1,r}$ is odd hole-free. Furthermore, observe that the only neighbour common to $x_2$ and $x_{2k + 1}$ is $x_1$. 
    Thus, there is no diamond in $P_4 \wprod K_{1,r}$ with a degree-2 vertex at $(3,1)$. But every vertex in an odd antihole on 7 or more vertices is a degree-2 vertex in some diamond, so $(3,1)$ cannot be a vertex of an odd antihole, and so $P_4 \wprod K_{1,r}$ is odd antihole-free. 
\end{proof}

\begin{cor}\label{cor:P4wprodK1uplusKr}
    The graph $P_4 \wprod (K_1 \uplus K_r)$ is perfect for any $r \geq 1$.
\end{cor}
\begin{proof}
    Lemma~\ref{lem:P4wprodK1r} along with Lemma~\ref{lem:wmp_comp_perf}.
\end{proof}

\subsubsection{Full case analysis}

We have finally gathered the required ingredients to prove Theorem~\ref{thm:main_intro}.
The proof constitutes a case analysis over all pairs of finite, simple graphs, which has been split into Lemmas~\ref{lem:chunk_i}-\ref{lem:chunk_viii}.
They are tied up in the proof of Theorem~\ref{thm:main}.

In this subsection we let the binary variable $z \in \{0,1\}$ be arbitrary and $\overline{z}$ be its complement. 
We do this for brevity, so we can make statements such as ``$G_z \wprod G_{\overline{z}}$ is perfect if and only if $G_z$ has property $\mathcal{P}_A$ and $G_{\overline{z}}$ has property $\mathcal{P}_B$, for any $z \in \{0,1\}$''.
The above statement is equivalent to the statements: ``$G_0 \wprod G_1$ is perfect if and only if either: $G_0$ has property $\mathcal{P}_A$ and $G_1$ has property $\mathcal{P}_B$, or $G_1$ has property $\mathcal{P}_A$ and $G_0$ has property $\mathcal{P}_B$''.

\begin{lem}
    \label{lem:chunk_i}
    Suppose $G_z \cong K_r \uplus K_s$ for $r + s \geq 3$ and $G_{\overline{z}}$ is paw-free. Then $G_z \wprod G_{\overline{z}}$ is perfect if and only if either
    \begin{enumerate}
        \item $G_z \cong K_2 \uplus E_1$ and $G_{\overline{z}} \cong C_5$; or
        \item $G_z \cong K_m \uplus K_1$ for $m \in \mathbb{N}$ and $G_{\overline{z}} \cong P_4$; or
        \item $G_{\overline{z}} \in \{ K_m, K_{m,n} \}$ for $m,n \in \mathbb{N}$; or
        \item $G_{\overline{z}}$ is a disjoint union of cliques and stars with two or more connected components.
    \end{enumerate} 
\end{lem}
\begin{proof}
   $G_{\overline{z}}$ is paw-free, so by Lemma~\ref{lem:paw_free_graphs}, it is a disjoint union of complete multipartite and triangle-free graphs. Moreover, by assumption $G_z$ has an induced $K_2 \uplus E_1$. 

   First, suppose that $G_{\overline{z}}$ has a triangle-free component $X$. Moreover, suppose that $X$ is bipartite and not complete. Then, by Lemma~\ref{lem:compbipP4} $X$ has an induced $P_4$. Lemma~\ref{lem:Bipartite_augments_P4_K2uE1} states that the weak modular product of a bipartite augment of $P_4$ with $K_2 \uplus E_1$ is not perfect, so if $G_{\overline{z}}\not\cong P_4$, $G_z \wprod G_{\overline{z}}$ is not perfect.\
   Suppose $G_{\overline{z}} \cong P_4$.
   If $r=1$ or $s=1$, then $G_z \wprod G_{\overline{z}} \cong (K_1 \uplus K_n) \wprod P_4$ for some $n$ and so is perfect by Corollary~\ref{cor:P4wprodK1uplusKr}, giving case (2). If $r \geq 2$ and $s \geq 2$, then $G_{z}$ has an induced $2K_2$ and $G_z \wprod G_{\overline{z}}$ is not perfect by Lemma~\ref{lem:P4_prods}.

   Suppose now that $X$ is nonbipartite, so has an odd hole by Lemma~\ref{lem:triangle_free_odd_hole}. Moreover, suppose the largest hole has length greater than 6. Then, $X$ has an induced $P_5$ and by Lemma~\ref{lem:K2uE1_products_nonperf} $G_z \wprod G_{\overline{z}}$ is not perfect. Now suppose that the largest odd hole in $X$ is a $C_5$. If $G_z \not\cong K_2\uplus E_1$ then $G_z$ contains either $2K_2$ or $K_3$ by assumption. Thus,  $G_z \wprod G_{\overline{z}}$ is not perfect by Lemma~\ref{lem:C5_prods}.
   If $G_z \cong (K_2 \uplus E_1)$, by Corollary~\ref{cor:C5prodK2uE1} and Lemma~\ref{lem:C5trifree_augments_C5_K2uE1} $G_z \wprod G_{\overline{z}}$ is perfect if and only if $G_{\overline{z}} \cong C_5$, giving case (1).

   Now suppose $G_{\overline{z}}$ is a disjoint union of complete multipartite graphs. If $G_{\overline{z}}$ has an induced diamond then $G_z \wprod G_{\overline{z}}$ is not perfect by Lemma~\ref{lem:K2uE1_products_nonperf}. By Lemma~\ref{lem:diamondfree_comp_multip}, a complete multipartite graph is diamond-free if it is a disjoint union of cliques and complete bipartites. Suppose $G_{\overline{z}}$ is diamond-free. If $G_{\overline{z}}$ is connected, then $G_{\overline{z}} \in \{ K_m, K_{m,n} \}$ for $m,n \in \mathbb{N}$ and $G_z \wprod G_{\overline{z}}$ is perfect by Lemmas~\ref{lem:GHdisj_cliqs}~and~\ref{lem:KruKs_Kmn}, giving case (3). Now suppose $G_{\overline{z}}$ has two or more connected components. If $G_{\overline{z}}$ has an induced $K_{2,2}$, then $G_z \wprod G_{\overline{z}}$ is not perfect by Lemma~\ref{lem:K2uE1_products_nonperf}; else $G_{\overline{z}}$ is a disjoint union of stars and cliques, in which case $G_z \wprod G_{\overline{z}}$ is perfect by Proposition~\ref{prop:twocliq_disjstrscliqs}, giving case (4).
\end{proof}

\begin{lem}
    \label{lem:chunk_ii}
    Suppose $G_z$ is a disjoint union of cliques with $k$ connected components, where $k \in \mathbb{N} \setminus \{ 2 \}$. Moreover, suppose $G_{\overline{z}}$ has an induced $P_3$. Then $G_z \wprod G_{\overline{z}}$ is perfect if and only if either
    \begin{enumerate}
        \item $G_z \in \{K_1, K_2\}$; or
        \item $G_z \cong K_r$ for $r \in \mathbb{N}$ and $G_{\overline{z}}$ is $(\text{odd hole, paw})$-free; or
        \item $G_z \cong E_k$ and $G_{\overline{z}}$ is $(\text{odd antihole, co-paw})$-free
    \end{enumerate}
\end{lem}
\begin{proof}
    Suppose $k\geq 3$. If $G_z$ is empty we get case (3) from Corollary~\ref{cor:GprodKn} and Lemma~\ref{lem:wmp_comp_perf}.
    Assume $G_z$ is nonempty. Then, $G_z$ has an induced $K_2 \uplus E_2$ and $G_z \wprod G_{\overline{z}}$ is not perfect by Lemma~\ref{lem:P3_prods}. Now assume $k=1$, i.e. $G_z \cong K_r$ for some $r \in \mathbb{N}$. Then we get cases (1) and (2) from Corollary~\ref{cor:GprodKn}.  
\end{proof}

\begin{lem}
    \label{lem:chunk_iii_i}
    Suppose $G_z$ is a disjoint union of complete multipartite graphs and triangle-free graphs. Moreover, suppose that $G_{\overline{z}}$ is connected, $(P_4, \text{cricket, dart, hourglass})$-free and contains an induced paw.
    Then $G_z \wprod G_{\overline{z}}$ is perfect if and only if $G_z \cong K_{m,n}$.
\end{lem}
\begin{proof}
    Suppose $G_{\bar{z}}$ has an induced $K_2 \uplus E_1$. Then, by Lemma~\ref{lem:K2uE1_products_nonperf} $G_z \wprod G_{\overline{z}}$ is not perfect.
    Now, assume $G_z$ is $(K_2 \uplus E_1)$-free. 
    By Corollary~\ref{cor:compmulti_K2uE1free}, a $(K_2 \uplus E_1)$-free graph is complete multipartite. Moreover, assume that $G_z$ has an induced diamond. $G_z \wprod G_{\overline{z}}$ contains $Y \wprod K_{1,1,2}$, so is not perfect by Lemma~\ref{lem:K2uE1_products_nonperf} as the paw contains $K_2 \uplus E_1$ as an induced subgraph. We now assume $G_z$ is diamond-free, so is a clique or complete bipartite by Lemma~\ref{lem:diamondfree_comp_multip}.
    If $G_z$ is a clique, then from Corollary~\ref{cor:GprodKn} $G_z \wprod G_{\overline{z}}$ is not perfect. If $G_z$ is complete bipartite then $G_z \wprod G_{\overline{z}}$ is perfect, from taking complements and using Proposition~\ref{prop:twocliq_disjstrscliqs} along with Lemmas~\ref{lem:disjunstrscliqcpml}~and~\ref{lem:wmp_comp_perf}.
\end{proof}

\begin{lem}
    \label{lem:chunk_iii_ii}
    Suppose $G_z$ is a disjoint union of complete multipartite graphs and triangle-free graphs, containing an induced $P_4$. Moreover, suppose that $G_{\overline{z}}$ is complete multipartite and contains an induced $P_3$.
    Then $G_z \wprod G_{\overline{z}}$ is perfect if and only if either
    \begin{enumerate}
        \item $G_z \cong C_5$, $G_{\overline{z}} \in \{ P_3, K_2 \}$; or
        \item $G_z \cong P_4$ , $G_{\overline{z}} \in \{ K_n, K_{1,n} \}$.
    \end{enumerate}
\end{lem}
\begin{proof}
    We denote by $X$ a component of $G_z$ containing an induced $P_4$. Suppose $X$ is bipartite. By Lemma~\ref{lem:Bipartite_augments_P4_P3}, if $G_z \not\cong P_4$ $G_z \wprod G_{\overline{z}}$ is not perfect. We thus assume $G_z \cong P_4$. If $G_{\overline{z}}$ has an induced diamond, $G_z \wprod G_{\overline{z}}$ is not perfect by Lemma~\ref{lem:P4_prods}. In accordance with Lemma~\ref{lem:diamondfree_comp_multip}, we thus let $G_{\overline{z}}$ be a clique or complete bipartite. If $G_{\overline{z}}$ is a clique $G_z \wprod G_{\overline{z}}$ is perfect by Corollary~\ref{cor:GprodKn}. If $G_{\overline{z}}$ is complete bipartite it either contains $K_{2,2}$ or is a star. In the former case $G_z \wprod G_{\overline{z}}$ is not perfect by Lemma~\ref{lem:P4_prods}; in the latter case is perfect by Lemma~\ref{lem:P4wprodK1r}. This gives us case (2).

    We now assume that $X$ is nonbipartite, so by Lemma~\ref{lem:triangle_free_odd_hole} $X$ contains an odd hole. If the largest odd hole in $X$ has length greater than 6, then $X$ has an induced $P_5$ and $G_z \wprod G_{\overline{z}}$ is not perfect by Lemma~\ref{lem:P3_prods}. 
    We thus assume $X$ contains $C_5$ as an induced subgraph. By Lemma~\ref{lem:C5trifree_augments_C5_P3}, if $G_z \not\cong C_5$, $G_z \wprod G_{\overline{z}}$ is not perfect. Now let $G_z \cong C_5$, which contains $P_4$ as an induced subgraph. Recall that $G_{\overline{z}}$ is complete multipartite. If $G_{\overline{z}}$ has an induced diamond, $G_z \wprod G_{\overline{z}}$ is not perfect by Lemma~\ref{lem:P4_prods}. With regard to Lemma~\ref{lem:diamondfree_comp_multip}, we thus let $G_{\overline{z}}$ be complete bipartite since $G_{\overline{z}}$ being a clique contradicts Lemma~\ref{lem:p3free}.
    Hence, $G_{\overline{z}}$ either contains $K_{2,2}$ or is a star. In the former case $G_z \wprod G_{\overline{z}}$ is not perfect by Lemma~\ref{lem:P4_prods}. In the latter, if $G_{\overline{z}} \cong K_{1,r}$ for $r \geq 3$ then $G_z \wprod G_{\overline{z}}$ is not perfect by Lemma~\ref{lem:C5_prods}. $C_5 \wprod K_{1,2}$ is perfect by Lemma~\ref{lem:C5perf_prods} and we have case (1).
\end{proof}

\begin{lem}
    \label{lem:chunk_iii_iii}
    Suppose $G_z$ is a disjoint union of complete multipartite graphs. Moreover, suppose that $G_{\overline{z}}$ is complete multipartite. Then $G_z \wprod G_{\overline{z}}$ is perfect if and only if either
    \begin{enumerate}
        \item $G_{\overline z} \cong K_n$
        \item $G_z \cong K_r \uplus K_s$, $ G_{\overline{z}} \cong K_{m,n}$; or
        \item $G_z$ and $G_{\overline{z}}$ are complete multipartite; or
    \end{enumerate}
\end{lem}
\begin{proof}
    If $G_z$ is connected, $G_z \wprod G_{\overline{z}}$ is perfect by Corollary~\ref{cor:completemultiprodperf} and we get case (3). We now assume $G_z$ has more than one connected component, so from Corollary~\ref{cor:compmulti_K2uE1free} $G_z$ has an induced $K_2 \uplus E_1$.
    If $G_{\overline{z}}$ has an induced diamond, $G_z \wprod G_{\overline{z}}$ is not perfect by Lemma~\ref{lem:K2uE1_products_nonperf}.
    Suppose now $G_{\overline{z}}$ is diamond-free.
    From Lemma~\ref{lem:diamondfree_comp_multip}, $G_{\overline{z}}$ is either a clique or complete bipartite. If $G_{\overline{z}}$ is a clique $G_z \wprod G_{\overline{z}}$ is perfect by Corollary~\ref{cor:GprodKn}, as $G_z$, being a disjoint union of complete multipartites, is $(\text{odd hole, paw})$-free. This falls into case (1). We assume now that
    $G_{\overline{z}} \cong K_{m,n}$. If $G_z$ has an induced diamond then $G_z$ has an induced $K_{1,1,2} \uplus E_1$, as it has more than one connected component, in which case $G_z \wprod G_{\overline{z}}$ is not perfect by Lemma~\ref{lem:P3_prods} (observe that $G_{\overline{z}}$ has an induced $P_3$ by Lemma~\ref{lem:p3free}).
    Now assume $G_z$ is a disjoint union of cliques and complete bipartites. If $G_z$ contains an induced $P_3$, it contains $P_3 \uplus E_1$ and $G_z \wprod G_{\overline{z}}$ is not perfect by Lemma~\ref{lem:P3_prods}. Finally, we let $G_z$ be a disjoint union of cliques. If $G_z \cong K_m \uplus K_n$, then $G_z \wprod G_{\overline{z}}$ is perfect by Lemma~\ref{lem:KruKs_Kmn}; else $G_z$ has three or more connected components and $G_z \wprod G_{\overline{z}}$ is not perfect, since $G_z$ contains $K_2 \uplus E_2$ as an induced subgraph and from Lemma~\ref{lem:P3_prods}, $P_3 \wprod (K_2 \uplus E_2)$ is not perfect. This gives us case (2) and completes the proof.   
\end{proof}

\begin{lem}
    \label{lem:chunk_iv}
    Suppose $G_{z}$ has an induced paw. Furthermore, suppose $G_{\overline{z}}$ is connected and $(P_4, \text{cricket, dart,}$ $\text{hourglass})$-free.  Then $G_z \wprod G_{\overline{z}}$ is perfect if and only if either
    \begin{enumerate}
        \item $G_{\overline{z}} \cong K_{m,n}$, $G_z$ is connected and $(P_4, \text{cricket, dart, hourglass})$-free; or
        \item $G_{\overline{z}} \in \{ K_1, K_2 \}$. 
    \end{enumerate}
\end{lem}
\begin{proof}
    First, suppose that $G_{\overline{z}}$ has an induced $K_2 \uplus E_1$, in which case $G_z \wprod G_{\overline{z}}$ is not perfect by Lemma~\ref{lem:K2uE1_products_nonperf}.
    From Corollary~\ref{cor:compmulti_K2uE1free}, a connected $(K_2 \uplus E_1)$-free graph is complete multipartite. We thus assume $G_{\overline{z}}$ is complete multipartite. If $G_{\overline{z}} \cong K_n$, we get case (2) from Corollary~\ref{cor:GprodKn}.
    We then let $G_{\overline{z}}\not\cong K_n$ for any $n\in \mathbb{N}$. Either $G_{\overline{z}}$ has an induced $K_{1,1,2}$ or is a clique or complete bipartite by Lemma~\ref{lem:diamondfree_comp_multip}. In the former case $G_z \wprod G_{\overline{z}}$ is not perfect by Lemma~\ref{lem:K2uE1_products_nonperf}, using the fact that $Y$ contains $K_2 \uplus E_1$ as an induced subgraph. In the latter case we have two scenarios: \emph{i.} $G_z$ is connected and $(P_4, \text{cricket, dart, hourglass})$-free, in which case $G_z \wprod G_{\overline{z}}$ is perfect from Proposition~\ref{prop:twocliq_disjstrscliqs} taken with Lemmas~\ref{lem:wmp_comp_perf}~and~\ref{lem:disjunstrscliqcpml}. In scenario \emph{ii.} $G_z$ is either disconnected or has an induced $P_4$, cricket, dart or hourglass. If $G_z$ is disconnected it contains an induced $P_3 \uplus E_1$. Recall that $G_{\overline{z}}$ contains an induced $P_3$ by Lemma~\ref{lem:p3free}, so $G_z \wprod G_{\overline{z}}$ is not perfect by Lemma~\ref{lem:P3_prods}. If $G_z$ contains any of $\{P_4, \text{cricket, dart, hourglass}\}$, then by Lemmas~\ref{lem:P4_prods}~and~\ref{lem:P3_cricket_dart_hourglass} $G_z \wprod G_{\overline{z}}$ is not perfect. This gives us case (1). 
\end{proof}

\begin{lem}
    \label{lem:chunk_v_i}
    Suppose $G_z$ has an induced $P_3$ and $\overline{G_{\overline{z}}} \cong r K_2$.
    Then, $G_z \wprod G_{\overline{z}}$ is perfect if and only if
    \begin{enumerate}
        \item $\overline{G_z}$ is a disjoint union of cliques (equiv. $G_z$ is complete multipartite); or
        \item $\overline{G_{\overline{z}}} \cong K_2$ (equiv. $G_{\overline{z}} \cong E_2$); or 
        \item $\overline{G_z}$ is a disjoint union of stars and cliques and $\overline{G_{\overline{z}}} \cong 2 K_2$ (equiv. $G_z$ is connected and $(P_4, \text{cricket, dart, hourglass})$-free, $G_{\overline{z}} \cong K_{2,2}$).
        \item $\overline{G_z} \cong K_{m,n}$ for $m,n \in \mathbb{N}$ and $\overline{G_{\overline{z}}} \cong 2 K_2$ (equiv. $G_z \cong K_m \uplus K_n$ and $G_{\overline{z}}\cong K_{2,2}$).
    \end{enumerate}
\end{lem}
\begin{proof}
    If $\overline{G_z}$ is a disjoint union of cliques, we have that $G_z$ and $G_{\overline{z}}$ are both complete multipartite by Lemma~\ref{lem:disjoint_compl_multi}. Corollary~\ref{cor:completemultiprodperf} gives us perfection of $G_z \wprod G_{\overline{z}}$, case (1). Now suppose $\overline{G_z}$ has an induced $P_3$. If $r \geq 3$, $G_z \wprod G_{\overline{z}}$ is not perfect by Lemma~\ref{lem:P3_prods} and Lemma~\ref{lem:wmp_comp_perf}. If $\overline{G_{\overline{z}}}\cong K_2$, we have case (2) from Corollary~\ref{cor:GprodKn}. Now suppose $\overline{G_{\overline{z}}}\cong 2K_2$. If we suppose that $\overline{G_{\overline{z}}}$ is paw-free we have case (3) and (4) from Lemma~\ref{lem:chunk_i}. If $\overline{G_{\overline{z}}}$ has an induced paw $G_z \wprod G_{\overline{z}}$ is not perfect by Lemma~\ref{lem:P3_prods}, as $G_{\overline{z}}$ has an induced $P_3 \uplus E_1$ (by Observation~\ref{obs:namedcomplements}).
\end{proof}

\begin{lem}
    \label{lem:chunk_v_ii}
    Suppose $G_z$ has an induced $P_3$ and is triangle-free. Moreover, suppose that $\overline{G_{\overline{z}}}$ is a disjoint union of stars with induced $P_3$.
    Then, $G_z \wprod G_{\overline{z}}$ is perfect if and only if
    \begin{enumerate}
        \item $G_z \cong C_5$, $\overline{G_{\overline{z}}} \cong P_3$ (equiv. $G_{\overline{z}} \cong K_2 \uplus E_1$); or
        \item $G_z \cong K_{m,n}$ for $m,n \geq 2$; or
        \item $G_z$ is a disjoint union of stars with two or more connected components, $\overline{G_{\overline{z}}} \cong K_{1,r}$ (equiv. $G_{\overline{z}} \cong K_1 \uplus K_r$); or
        \item $G_z \cong K_{1,r}$; or
        \item $G_z \cong P_4$, $\overline{G_{\overline{z}}} \cong K_{1,r}$ (equiv. $G_{\overline{z}} \cong K_1 \uplus K_r$). 
    \end{enumerate}
\end{lem}
\begin{proof}
    First suppose that $G_z$ is nonbipartite, so by Lemma~\ref{lem:triangle_free_odd_hole} it has an odd hole. If the largest odd hole in $G_z$ has size six or larger, then $G_z$ has an induced $P_5$.
    By Lemmas~\ref{lem:p3free},~\ref{lem:disjoint_compl_multi} and Corollary~\ref{cor:compmulti_K2uE1free}, $G_{\overline{z}}$ has an induced $P_3$ if and only if $\overline{G_{\overline{z}}}$ is disconnected.
    Thus, $G_z \wprod G_{\overline{z}}$ is not perfect by Lemma~\ref{lem:P3_prods} if $\overline{G_{\overline{z}}}$ is disconnected.
    Let $\overline{G_{\overline{z}}}$ be connected. Then, $G_{\overline{z}} \cong K_r \uplus K_1$ for some $r \geq 2$ and so $G_{\overline{z}}$ contains an induced $K_2 \uplus E_1$.
    In this case $G_z \wprod G_{\overline{z}}$ is not perfect by Lemma~\ref{lem:Bipartite_augments_P4_K2uE1} as $P_5$ is a bipartite augment of $P_4$.

    Now suppose $G_z$ has an induced $C_5$. If $G_z \not\cong C_5$, then $G_z \wprod G_{\overline{z}}$ is not perfect by Lemma~\ref{lem:C5trifree_augments_C5_P3}. We now let $G_z \cong C_5$, in which case $\overline{G_z} \cong C_5$. If $\overline{G_{\overline{z}}} \not\cong K_{1,2}$, then either it contains an induced $P_3 \uplus E_1$, or it contains an induced $K_{1,3}$.
    By Lemmas~\ref{lem:P3_prods}~and~\ref{lem:C5_prods} respectively, $G_z \wprod G_{\overline{z}}$ is not perfect in both cases.
    If $\overline{G_{\overline{z}}} \cong K_{1,2}$, $G_{\overline{z}}\cong K_2 \uplus E_1$ and $G_z \wprod G_{\overline{z}}$ is perfect by Corollary~\ref{cor:C5prodK2uE1}, giving case (1).
    
    Now we suppose $G_z$ is bipartite. Suppose also that it has an induced $P_4$, so has a connected component that is not complete bipartite by Lemma~\ref{lem:compbipP4}. Moreover, $\overline{G_z}$ also has an induced $P_4$ since $P_4$ is self-complementary. Now assume that $\overline{G_{\overline{z}}}$ has two or more connected components, in which case it has an induced $P_3 \uplus E_1$ and thus $G_{\overline{z}}$ has an induced paw by Observation~\ref{obs:namedcomplements}. Then, $G_z \wprod G_{\overline{z}}$ is not perfect by Lemma~\ref{lem:P4_prods}. Now suppose that $\overline{G_{\overline{z}}}$ is connected, so $\overline{G_{\overline{z}}} \cong K_{1,r}$ and $G_{\overline{z}} \cong K_r \uplus E_1$. If $\overline{G_z}\not\cong P_4$, by Lemmas~\ref{lem:wmp_comp_perf}~and~\ref{lem:Bipartite_augments_P4_P3} $G_z \wprod G_{\overline{z}}$ is not perfect. If $\overline{G_z}\cong P_4$, $G_z \cong P_4$ and $G_z \wprod G_{\overline{z}}$ is perfect by Corollary~\ref{cor:P4wprodK1uplusKr}, giving case (5).

    Now we suppose $G_z$ is a disjoint union of complete bipartites. First suppose $G_z$ has an induced $K_{2,2}$. If it has two or more connected components then $G_z$ contains an induced $K_{2,2} \uplus E_1$ and $G_z \wprod G_{\overline{z}}$ contains an induced $P_3 \wprod (P_3 \uplus E_1)$, so is not perfect by Lemma~\ref{lem:P3_prods}. If $G_z$ is connected, then $G_z \cong K_{m,n}$, $\overline{G_z} \cong K_m \uplus K_n$. By Proposition~\ref{prop:twocliq_disjstrscliqs} and Lemma~\ref{lem:wmp_comp_perf} $G_z \wprod G_{\overline{z}}$ is perfect, giving case (2). 
    Now let $G_z$ be a disjoint union of stars. Moreover, suppose $G_z$ has two or more connected components so has induced $2K_2$; equivalently $\overline{G_z}$ has an induced $K_{2,2}$. Also, let $\overline{G_{\overline{z}}}$ have two or more connected components so has induced $P_3 \uplus E_1$. From Lemma~\ref{lem:K22_prods} we have that $K_{2,2} \wprod (P_3 \uplus E_1)$ is not perfect, and so by Lemma~\ref{lem:wmp_comp_perf} $G_z \wprod G_{\overline{z}}$ is not perfect. If $\overline{G_{\overline{z}}}$ is connected $G_{\overline{z}} \cong K_1 \uplus K_r$, so by Proposition~\ref{prop:twocliq_disjstrscliqs} and Lemma~\ref{lem:wmp_comp_perf} $G_z \wprod G_{\overline{z}}$ is perfect, giving case (3).
    Finally, suppose $G_z$ is connected. Then, $G_z \cong K_{1,r}$, $\overline{G_z} \cong K_1 \uplus K_r$ and $G_z \wprod G_{\overline{z}}$ is perfect by Proposition~\ref{prop:twocliq_disjstrscliqs} and Lemma~\ref{lem:wmp_comp_perf}, giving case (4). This completes the proof.
\end{proof}

\begin{lem}\label{lem:K2uE2_free_P3}
    Let $G$ be a $(K_2 \uplus E_2)$-free graph such that $\alpha(G) \geq 3$. Then $G \wprod P_3$ is perfect if and only if either: $G$ is connected and $(P_4, \text{cricket, dart, hourglass})$-free; or $G\cong E_n$.
\end{lem}
\begin{proof}
    $(\Rightarrow)$ If $G \cong E_n$ then $G \wprod P_3$ is perfect, since $E_n \wprod P_3 \cong K_n \wprod (K_2 \uplus E_1)$ and the latter is perfect by Corollary~\ref{cor:GprodKn}. Moreover, one can see that if $G$ is connected and $(P_4, \text{cricket, dart, hourglass})$-free, $G \wprod H$ is perfect by taking complements and using Lemmas~\ref{lem:wmp_comp_perf}~and~\ref{lem:disjunstrscliqcpml} with Proposition~\ref{prop:twocliq_disjstrscliqs}.

    $(\Leftarrow)$ We prove the contrapositive, namely that $G \wprod P_3$ is not perfect, where $G$ is a $(K_2 \uplus E_2)$-free graph such that $\alpha(G) \geq 3$ that is disconnected, or contains an induced $P_4$, $\text{cricket}$, $\text{dart}$ or $\text{hourglass}$.
    If $G \cong P_4$, then $\alpha(G) = 2$ and we have a contradiction. If $G$ is an augment of $P_4$, then $G \wprod P_3$ is not perfect unless $G \cong C_5$, in which case $\alpha(G) =2$ and we have a contradiction. Thus, for any $G$ satisfying the conditions of the proposition containing an induced $P_4$, $G \wprod P_3$ is not perfect.
    For $X \in \{\text{cricket, dart, hourglass}\}$, $X \wprod P_3$ is not perfect by Lemma~\ref{lem:P3_cricket_dart_hourglass}, so for any $G$ such that $X \subseteq G$, $G \wprod P_3$ is not perfect. Suppose $G$ is disconnected and nonempty.
    Then, if $G$ has more than three connected components, it has an induced $K_2 \uplus E_2$, a contradiction.
    So $G$ must have two components. If both components of $G$ are cliques then $\alpha(G) = 2$ and we have a contradiction, so at least one components of $G$ contains a $P_3$. Now, since $(P_3 \uplus E_1) \wprod P_3$ is not perfect from Lemma~\ref{lem:P3_prods}, $G$ is not perfect and the lemma is proved.
\end{proof}

\begin{lem}
    \label{lem:chunk_v_iii}
    Suppose $G_z$ has an induced $P_3$ and contains a triangle. Moreover, suppose that $\overline{G_{\overline{z}}}$ is a disjoint union of stars with induced $P_3$.
    Then, $G_z \wprod G_{\overline{z}}$ is perfect if and only if $\overline{G_z}$ is connected and $(P_4, \text{cricket, dart, hourglass})$-free and $\overline{G_{\overline{z}}} \cong K_{1,r}$ (equiv. $G_z$ is a disjoint union of stars and cliques and $G_{\overline{z}} \cong K_1 \uplus K_r$).
\end{lem}
\begin{proof}
    Suppose $\overline{G_z}$ has an induced $K_2 \uplus E_2$. Then, $G_z \wprod G_{\overline{z}}$ is not perfect by Lemmas~\ref{lem:wmp_comp_perf}~and~\ref{lem:P3_prods}. Now suppose that $\overline{G_z}$ is $(K_2 \uplus E_2)$-free. By Lemma~\ref{lem:wmp_comp_perf} and Lemma~\ref{lem:K2uE2_free_P3} $G_z \wprod G_{\overline{z}}$ is not perfect if $\overline{G_z}$ is disconnected or contains an induced $P_4$, cricket, dart or hourglass.
    We thus assume that $\overline{G_z}$ is connected and $(P_4, \text{cricket, dart, hourglass})$-free. Now if $\overline{G_{\overline{z}}}$ is connected we have that $G_{\overline{z}} \cong K_1 \uplus K_r$ and $G_z$ is a disjoint union of cliques and stars by Lemma~\ref{lem:disjunstrscliqcpml}, and so by Proposition~\ref{prop:twocliq_disjstrscliqs} $G_z \wprod G_{\overline{z}}$ is perfect, yielding case (1).
    Now assume $\overline{G_{\overline{z}}}$ is disconnected. If $\overline{G_{z}} \cong K_n$, then $G_z \cong E_n$ contradicting the assumptions of the lemma.
    Finally, suppose $\overline{G_{z}} \not\cong K_n$, in which case $\overline{G_z}$ has an induced $P_3$ by Lemma~\ref{lem:p3free}.
    Recall that $\overline{G_{\overline{z}}}$ has an induced $P_3 \uplus E_1$. From Lemmas~\ref{lem:wmp_comp_perf}~and~\ref{lem:P3_prods} $G_z \wprod G_{\overline{z}}$ is not perfect. 
\end{proof}

\begin{lem}
    \label{lem:chunk_vi}
    Suppose $G_z$ has an induced $P_3$ and $\overline{G_{\overline{z}}}$ is a disjoint union of complete bipartites, containing an induced $K_{2,2}$.
    Then $G_z \wprod G_{\overline{z}}$ is perfect if and only if $\overline{G_z}$ is connected and $(P_4, \text{cricket, dart, hourglass})$-free and $\overline{G_{\overline{z}}} \cong K_{m,n}$ (equiv. $G_z$ is a disjoint union of stars and cliques, $G_{\overline{z}} \cong K_m \uplus K_n$).
\end{lem}
\begin{proof}
    Observe that $\overline{G_z}$ has an induced $K_2 \uplus E_1$. Now suppose $\overline{G_{\overline{z}}}$ has two or more connected components. Then, $\overline{G_{\overline{z}}}$ has an induced $K_{2,2} \uplus E_1$. Thus, by Lemmas~\ref{lem:K2uE1_products_nonperf}~and~\ref{lem:wmp_comp_perf} $G_z \wprod G_{\overline{z}}$ is not perfect. Now suppose $\overline{G_{\overline{z}}}$ is connected, i.e. $\overline{G_{\overline{z}}} \cong K_{m,n}$. Moreover, suppose that $\overline{G_z}$ is not connected. Either $\overline{G_z}$ is a disjoint union of cliques or has an induced $P_3 \uplus E_1$, by Lemma~\ref{lem:p3free}. In the former case, we contradict Lemma~\ref{lem:p3free}. In the latter case Lemmas~\ref{lem:P3_prods}~and~\ref{lem:wmp_comp_perf} give that $G_z \wprod G_{\overline{z}}$ is not perfect.

    Now suppose that $\overline{G_z}$ is connected. If $\overline{G_z}$ contains an induced $P_4$, $G_z \wprod G_{\overline{z}}$ is not perfect by Lemmas~\ref{lem:P4_prods}~and~\ref{lem:wmp_comp_perf}. Let $\overline{G_z}$ be $P_4$-free. If $\overline{G_z}$ contains an induced cricket, dart, or hourglass then by Lemmas~\ref{lem:P3_cricket_dart_hourglass}~and~\ref{lem:wmp_comp_perf} $G_z \wprod G_{\overline{z}}$ is not perfect, since $K_{2,2}$ contains $P_3$. Otherwise, Lemma~\ref{lem:wmp_comp_perf} and Proposition~\ref{prop:twocliq_disjstrscliqs} give that $G_z \wprod G_{\overline{z}}$ is perfect.
\end{proof}

\begin{lem}
    \label{lem:chunk_vii}
    Suppose $G_z$ has an induced $P_3$ and $\overline{G_{\overline{z}}}$ is bipartite with induced $P_4$. Then $G_z \wprod G_{\overline{z}}$ is perfect if and only if $\overline{G_{\overline{z}}}\cong P_4$ and $\overline{G_z} \in \left\{ C_5, P_4, K_1 \uplus K_s \right\}$ (equiv. $G_{\overline{z}} \cong P_4$ and $G_z \in \left\{ C_5, P_4, K_{1, s} \right\}$).
\end{lem}
\begin{proof}
    Observe that $\overline{G_z}$ has an induced $K_2 \uplus E_1$. Let us suppose that $\overline{G_{\overline{z}}} \not\cong P_4$. Then, by Lemmas~\ref{lem:Bipartite_augments_P4_K2uE1}~and~\ref{lem:wmp_comp_perf} $G_z \wprod G_{\overline{z}}$ is not perfect. Thus, we suppose that $\overline{G_{\overline{z}}} \cong P_4$. Suppose that $\overline{G_z}$ has an induced paw. Then, by Lemmas~\ref{lem:P4_prods}~and~\ref{lem:wmp_comp_perf} $G_z \wprod G_{\overline{z}}$ is not perfect. Now let $\overline{G_z}$ be paw-free, that is, by Lemma~\ref{lem:paw_free_graphs} $\overline{G_z}$ is a disjoint union of complete multipartite and triangle-free graphs.

    First, assume that $\overline{G_z}$ has triangle-free component $X$, that is not complete bipartite. If $X$ is bipartite, $G_z \wprod G_{\overline{z}}$ is perfect if and only if $\overline{G_z} \cong P_4$, from Lemmas~\ref{lem:Bipartite_augments_P4_P3},~\ref{lem:C5perf_prods}~and~\ref{lem:wmp_comp_perf}. Now suppose $X$ is nonbipartite. By Lemma~\ref{lem:triangle_free_odd_hole} it contains an odd hole. If the largest odd hole in $X$ has seven or more vertices, $X$ contains an induced $P_5$ and so by Lemmas~\ref{lem:P3_prods}~and~\ref{lem:wmp_comp_perf} $G_z \wprod G_{\overline{z}}$ is not perfect. Now suppose the largest odd hole in $X$ is a $C_5$. By Lemmas~\ref{lem:C5trifree_augments_C5_P3},~\ref{lem:C5perf_prods}~and~\ref{lem:wmp_comp_perf} $G_z \wprod G_{\overline{z}}$ is perfect if and only if $\overline{G_z} \cong C_5$.

    Now assume $\overline{G_z}$ is a disjoint union of complete multipartites. $\overline{G_z}$ is disconnected, for otherwise; by Lemma~\ref{cor:compmulti_K2uE1free} it is $(K_2 \uplus E_1)$-free. However, $\overline{G_z}$ has an induced $K_2 \uplus E_1$ by assumption. If $\overline{G_z}$ has three or more components, it contains an induced $K_2 \uplus E_2$, and so by Lemmas~\ref{lem:P4_prods}~and~\ref{lem:wmp_comp_perf} $G_z \wprod G_{\overline{z}}$ is not perfect. We thus assume $\overline{G_z}$ has two connected components. By Lemma~\ref{lem:diamondfree_comp_multip} $\overline{G_z}$ has an induced diamond or is a disjoint union of cliques and complete bipartite graphs. In the former case, $G_z \wprod G_{\overline{z}}$ is not perfect by Lemmas~\ref{lem:P4_prods}~and~\ref{lem:wmp_comp_perf}. We now assume the latter. If $\overline{G_z}$ contains $K_{2,2}$ as an induced subgraph, $G_z \wprod G_{\overline{z}}$ is not perfect by Lemmas~\ref{lem:P4_prods}~and~\ref{lem:wmp_comp_perf}. Moreover, if $\overline{G_z}$ contains a star (that is not also a clique) then it contains $P_3$ and thus $P_3 \uplus E_1$, since it has two components. Lemmas~\ref{lem:P3_prods}~and~\ref{lem:wmp_comp_perf} give us that $G_z \wprod G_{\overline{z}}$ is not perfect in this case. Assume $\overline{G_z}$ is a disjoint union of cliques. If $\overline{G_z}$ contains $2K_2$, $G_z \wprod G_{\overline{z}}$ is not perfect from Lemmas~\ref{lem:P4_prods}~and~\ref{lem:wmp_comp_perf}. This leaves $\overline{G_z} \cong K_1 \uplus K_s$, in which case $G_z \wprod G_{\overline{z}}$ is perfect from Corollary~\ref{cor:P4wprodK1uplusKr}. 
\end{proof}

\begin{lem}
    \label{lem:chunk_viii}
    Suppose $G_z$ has an induced $P_3$ and $\overline{G_{\overline{z}}}$ is nonbipartite and triangle-free. Then, $G_z \wprod G_{\overline{z}}$ is perfect if and only if $\overline{G_{\overline{z}}}\cong C_5$ and $\overline{G_z} \in \left\{ K_2 \uplus E_1, P_4, C_5 \right\}$ (equiv. $G_{\overline{z}} \cong C_5$ and $G_z \in \left\{ P_3, P_4, C_5 \right\}$).
\end{lem}
\begin{proof}
    Note that $\overline{G_z}$ has an induced $K_2 \uplus E_1$. By Lemma~\ref{lem:triangle_free_odd_hole}, $\overline{G_{\overline{z}}}$ has an odd hole. If the largest odd hole has seven or more vertices, $\overline{G_{\overline{z}}}$ contains an induced $P_5$. Then, by Lemmas~\ref{lem:K2uE1_products_nonperf}~and~\ref{lem:wmp_comp_perf} $G_z \wprod G_{\overline{z}}$ is not perfect. Now suppose $\overline{G_{\overline{z}}}$ contains an induced $C_5$. By Lemmas~\ref{lem:C5trifree_augments_C5_K2uE1}~and~\ref{lem:wmp_comp_perf}, $G_z \wprod G_{\overline{z}}$ is not perfect if $\overline{G_{\overline{z}}} \not\cong C_5$. So we assume $\overline{G_{\overline{z}}} \cong C_5$, in which case $G_{\overline{z}} \cong C_5$. Then, Lemmas~\ref{lem:C5perf_prods}~and~\ref{lem:wmp_comp_perf} give the result. 
\end{proof}

\setcounter{thm}{0}
\begin{thm}[Restated for convenience]
    \label{thm:main}
    The graph $G = G_0 \wprod G_1$ is perfect if and only if one of the following holds:
    \begin{enumerate}
     \item $G_z \in \{K_1, K_2, E_2\}$, $G_{\overline{z}}$ arbitrary;
     \item $G_z \cong P_4$, $G_{\overline{z}} \in \{ K_{1,r}, K_{r} \uplus K_1, P_4 \}$;
     \item $G_z \cong C_5$, $G_{\overline{z}} \in \{ P_3, K_2 \uplus E_1, P_4, C_5 \}$;
     \item $G_z \cong K_r \uplus K_s$, $G_{\overline{z}}$ is a disjoint union of stars and cliques;
     \item $G_z \cong K_{m,n}$, $G_{\overline{z}}$ is connected and $(P_4, \text{cricket}, \text{dart}, \text{hourglass})$-free;
     \item $G_z \cong K_n$, $G_{\overline{z}}$ $(\text{odd hole, paw})$-free;
     \item $G_z \cong E_n$, $G_{\overline{z}}$ $(\text{odd antihole, co-paw})$-free;
     \item $G_z$, $G_{\overline{z}}$ are complete multipartite;
     \item $G_z$, $G_{\overline{z}}$ are disjoint unions of cliques;
     \item $G_z \cong K_r \uplus K_s$, $G_{\overline{z}} \cong K_{m,n}$;
    \end{enumerate}
    for any $m,n,r,s,z$, where $m,n,r,s \in \mathbb{N}$, and $z \in \{0,1\}$, with its (Boolean) negation denoted by $\overline{z}$.
\end{thm}
\begin{proof}
    $(\Rightarrow)$ We prove the forward direction for each case in turn:
    Case (1) follows directly from Corollary~\ref{cor:GprodKn} and Lemma~\ref{lem:wmp_comp_perf}. For case (2), perfection of $P_4 \wprod P_4$ follows as a direct corollary of Lemma~\ref{lem:C5perf_prods}, as $C_5$ contains $P_4$ as an induced subgraph. $P_4 \wprod K_{1,r}$ and $P_4 \wprod (K_1 \uplus K_r)$ are perfect by Lemma~\ref{lem:P4wprodK1r} and Corollary~\ref{cor:P4wprodK1uplusKr} respectively.
    Case (3) follows directly from Lemma~\ref{lem:C5perf_prods} and Corollary~\ref{cor:C5prodK2uE1}.
    Proposition~\ref{prop:twocliq_disjstrscliqs} gives case (4).
    Proposition~\ref{prop:twocliq_disjstrscliqs} taken with Lemmas~\ref{lem:wmp_comp_perf}~and~\ref{lem:disjunstrscliqcpml} give case (5).
    Corollary~\ref{cor:GprodKn} yields (6).
    For case (7), combine Corollary~\ref{cor:GprodKn} and Lemma~\ref{lem:wmp_comp_perf}.
    Corollary~\ref{cor:completemultiprodperf} gives (8).
    Lemma~\ref{lem:GHdisj_cliqs} yields (9).
    Finally, (10) is proven by Lemma~\ref{lem:KruKs_Kmn}.    
    
    $(\Leftarrow)$ 
    We show that the weak modular product of any pair of graphs not falling into cases (1)--(10) is not perfect.
    First suppose $G_z$ is $P_3$-free. Then, by Lemma~\ref{lem:p3free} $G_z$ is a disjoint union of cliques.
    If $G_{\overline{z}}$ is a disjoint union of cliques then we have case (9).
    Thus we suppose $G_{\overline{z}}$ is not a disjoint union of cliques, so has an induced $P_3$ by Lemma~\ref{lem:p3free}.
    If $G_z$ has $k$ connected components, where $k\in \mathbb{N} \setminus \{2\}$, by Lemma~\ref{lem:chunk_ii}, the only cases where $G_z \wprod G_{\overline{z}}$ is perfect belong to cases (1), (6) and (7). 
    If $k = 2$, either $G_z \cong E_2$, and we have case (1), or $G_z$ is nonempty and contains an induced $K_2 \uplus E_1$.
    If $G_{\overline{z}}$ has an induced paw, then $G_z \wprod G_{\overline{z}}$ is not perfect by Lemma~\ref{lem:K2uE1_products_nonperf}.
    If $G_{\overline{z}}$ is paw-free then by Lemma~\ref{lem:chunk_i} the only pairs which give a perfect product fall under cases (3), (2), (10) and (4).
    This completes the proof for the case when $G_z$ is $P_3$-free.

    Now, suppose $G_z$ contains an induced $P_3$. If $G_{\overline{z}}$ is empty we have a perfect product only under the conditions of case (7), so assume $G_{\overline{z}}$ is nonempty.
    Furthermore, suppose $\alpha(G_{\overline{z}}) \geq 3$. By Lemma~\ref{lem:P3_prods}, $G_z \wprod G_{\overline{z}}$ is not perfect if $G_{\overline{z}}$ has an induced $K_2 \uplus E_2$, as $G_z$ contains an induced $P_3$. We thus consider the case when $G_{\overline{z}}$ is $(K_2 \uplus E_2)$-free.
    By Lemma~\ref{lem:K2uE2_free_P3}, $G_z \wprod G_{\overline{z}}$ is not perfect if $G_{\overline{z}}$ is disconnected or contains an induced $P_4$, cricket, dart, or hourglass.
    So we assume $G_{\overline{z}}$ is connected and $(P_4, \text{cricket, dart, hourglass})$-free.
    Now, if $G_z$ has an induced paw, by Lemma~\ref{lem:chunk_iv} the only cases when $G_z \wprod G_{\overline{z}}$ is perfect belong to cases (5) and (1).
    We thus assume $G_z$ is paw free, so is a disjoint union of complete multipartite and triangle-free graphs by Lemma~\ref{lem:paw_free_graphs}.
    If $G_{\overline{z}} \cong K_n$, it is $P_3$-free by Lemma~\ref{lem:p3free}. Then, we have a perfect product if and only if the conditions of case (6) are satisfied.
    Suppose $G_{\overline{z}}$ has induced $P_3$. Furthermore, if $G_{\overline{z}}$ has an induced paw, by Lemma~\ref{lem:chunk_iii_i} $G_z \wprod G_{\overline{z}}$ is perfect only in case (5).
    Now let $G_{\overline{z}}$ be paw-free, so by Lemma~\ref{lem:p4pawfree} is complete multipartite.
    We now suppose $G_z$ has a triangle-free component $X$ that is not complete bipartite, and so has an induced $P_4$ by Lemmas~\ref{lem:compbipP4}~and~\ref{lem:triangle_free_odd_hole}. We then have $G_z \wprod G_{\overline{z}}$ is perfect only in cases (3), (1) and (2) from Lemma~\ref{lem:chunk_iii_ii}.
    Assume $G_z$ is $P_4$-free, so by Lemma~\ref{lem:p4pawfree} is a disjoint union of complete multipartites.
    Then, from Lemma~\ref{lem:chunk_iii_iii} $G_z \wprod G_{\overline{z}}$ only in cases (6), (10) and (8).
    This covers the case when $G_z$ contains an induced $P_3$ and $G_{\overline{z}}$ is nonempty with $\alpha(G_{\overline{z}}) \geq 3$.

    We now consider the case when $G_z$ contains an induced $P_3$ and $G_{\overline{z}}$ is nonempty with $\alpha(G_{\overline{z}}) \leq 2$. If $\alpha(G_{\overline{z}}) = 1$, $G_{\overline{z}} \cong K_n$ and by Corollary~\ref{cor:GprodKn} is perfect if and only if $G_z$ is $(\text{odd hole, paw})$-free, falling into case (6). We are left with $\alpha(G_{\overline{z}}) = 2$ or, equivalently, $\omega(\overline{G_{\overline{z}}}) = 2$.
    Since $\omega(\overline{G_{\overline{z}}}) = 2$, $\overline{G_{\overline{z}}}$ is triangle-free. Let us first consider the case when $\overline{G_{\overline{z}}}$ is nonbipartite. Then, by Lemma~\ref{lem:chunk_viii} $G_z \wprod G_{\overline{z}}$ is only perfect in case (3).
    Now suppose $\overline{G_{\overline{z}}}$ is bipartite. If $\overline{G_{\overline{z}}}$ contains an induced $P_4$, by Lemma~\ref{lem:chunk_vii} $G_z \wprod G_{\overline{z}}$ is only perfect in cases (2) and (3). 
    Thus we suppose $\overline{G_{\overline{z}}}$ is a disjoint union of complete bipartite graphs, as from Lemma~\ref{lem:compbipP4} these are the only bipartite graphs that are $P_4$-free. 
    Now suppose $\overline{G_{\overline{z}}}$ has an induced $K_{2,2}$. Then, by Lemma~\ref{lem:chunk_vi}, $G_z \wprod G_{\overline{z}}$ is only perfect in case (4).  
    We thus assume that $\overline{G_{\overline{z}}}$ is a disjoint union of stars. If $\overline{G_{\overline{z}}} \cong rK_2$, by Lemma~\ref{lem:chunk_v_i}, $G_z \wprod G_{\overline{z}}$ is only perfect in cases (9), (1), (5) and (10). Now assume that $\overline{G_{\overline{z}}}$ is a disjoint union of stars with an induced $P_3$. Since $G_z$ has an induced $P_3$, $\overline{G_z}$ is nonempty. We then consider three scenarios: \emph{i.}  $\alpha(\overline{G_z}) = 1$, \emph{ii.} $\alpha(\overline{G_z}) = 2$ and \emph{iii.} $\alpha(\overline{G_z}) \geq 3$. In scenario \emph{i.}, $G_z \cong K_n$, yet $G_z$ has an induced $P_3$ by assumption, contradicting Lemma~\ref{lem:p3free}. In scenario \emph{ii.}, $\omega(G_z) = 2$ and so $G_z$ is triangle-free. Lemma~\ref{lem:chunk_v_ii} gives us that $G_z \wprod G_{\overline{z}}$ is only perfect in cases (3), (5), (4) and (2). Finally, in scenario \emph{iii.}, $\omega(G_z) \geq 3$ so $G_z$ contains a triangle. From Lemma~\ref{lem:chunk_v_iii}, $G_z \wprod G_{\overline{z}}$ is only perfect in case (4).

    This completes the proof as we have enumerated all pairs of graphs.
    \end{proof}

\section{Discussion}

Theorem~\ref{thm:main} gives us a characterisation of all pairs of graphs whose weak modular product is perfect.
In light of the discussion of Section~\ref{sec:wmp_and_iso}, it is natural to ask if any of the cases (1)--(10) fall into classes of graphs for which there is no existing efficient graph isomorphism algorithm.
We analyse each case in turn. Cases (1)--(3) admit trivial algorithms to check isomorphism.
In case (4), there is a simple algorithm: find the connected components, in the case when the disjoint union of stars and cliques has two connected components, count the neighbours of each vertex and compare; otherwise the graphs are trivially non-isomorphic.
In case (5), one can use use the previous algorithm after taking complements, or, observe that the two graphs in question are cographs and so admit an efficient GI algorithm~\cite{Corneil1981}.
Cases (6) and (7) admit trivial algorithms by counting vertex neighbours.
Cases (8) and (9) are cographs so have an efficient algorithm.
Case (10) is trivial by counting connected components.
Thus, this technique does not lead to an algorithm for GI on any new graph families.

\section*{Acknowledgements}

Thanks to Simone Severini for providing the initial question directing this work, and to anonymous reviewers for commentary on a previous version of the manuscript.
This research was funded by the Engineering and Physical Sciences Research Council (EPSRC).
The software SageMath~\cite{sagemath} has been invaluable to this work, alongside the ISGCI graph database~\cite{graphclasses}.
\bibliographystyle{nicebib-alpha}
\bibliography{graphiso.bib,extrarefs.bib}
\end{document}